\documentclass[11pt]{amsart}
\usepackage[utf8]{inputenc}
\usepackage[T1]{fontenc}
\usepackage[english]{babel}

\usepackage[left=3cm,right=3cm,top=3cm,bottom=3cm]{geometry}

\usepackage{mlmodern}
\usepackage{eucal}

\usepackage{amsmath}
\usepackage{amsthm}
\usepackage{amssymb}
\usepackage{mathtools}
\usepackage{stmaryrd}
\usepackage{mathrsfs}  
\usepackage{dsfont}  

\usepackage{chngcntr}

\makeatletter
\def\l@subsection{\@tocline{2}{0pt}{2.5pc}{5pc}{}}
\makeatother

\newcommand{\nocontentsline}[3]{}
\newcommand{\tocless}[2]{\bgroup\let\addcontentsline=\nocontentsline#1{#2}\egroup}

\let\origcontentsline\addcontentsline
\newcommand\stoptoc{\let\addcontentsline\nocontentsline}
\newcommand\resumetoc{\let\addcontentsline\origcontentsline}

\DeclareRobustCommand{\SkipTocEntry}[5]{}

\makeatletter
\renewcommand{\l@section}{\@tocline{1}{0pt}{10pt}{1pc}{\bfseries}}
\makeatother

\usepackage[dvipsnames]{xcolor}
\usepackage[all]{xy}
\usepackage{tikz-cd}
\usepackage{array}
\usetikzlibrary{shapes}
\usetikzlibrary[patterns]
\usepackage[tableposition=below]{caption}
\usepackage{subcaption}
\usepackage{lscape}  

\usepackage[colorlinks=true, citecolor=MidnightBlue, linkcolor=MidnightBlue]{hyperref}
\usepackage{cleveref}

\usepackage{enumitem}
\setitemize{label=$\triangleright$}
\setlist{topsep=8pt, itemsep=4pt }

\usepackage{todonotes}

\makeatletter
\def\thm@space@setup{%
  \thm@preskip=4\topsep \thm@postskip=\thm@preskip
}
\makeatother

 \newtheorem{theo}{Theorem}[section]
 \newtheorem{prop}[theo]{Proposition}
 \newtheorem{defi}[theo]{Definition}
 \newtheorem{lem}[theo]{Lemma}
 \newtheorem{coro}[theo]{Corollary}
 \newtheorem{conj}[theo]{Conjecture}

 \newtheorem*{namedtheorem}{\theoremname}
 \newcommand{\theoremname}{Theorem}
 \newenvironment{named}[1]{\renewcommand{\theoremname}{#1}\begin{namedtheorem}}{\end{namedtheorem}}

 \theoremstyle{remark}
 
 \newtheorem{rk}[theo]{Remark}
 \newtheorem{ex}[theo]{Example}

 \newenvironment{demo}{\begin{proof}}{\end{proof}}

\newcommand{\floor}{\node[draw,ellipse, minimum width=1cm, minimum height = 0.6 cm]}

\def\mfloor (#1) at (#2,#3) {
    \node[draw,ellipse, minimum width=1cm, minimum height = 0.6 cm] (#1) at (#2,#3) {$\bullet$} ;
}

\def\ufloor (#1) at (#2,#3) (#4) {
    \node[draw,ellipse, minimum width=1cm, minimum height = 0.6 cm] (#1) at (#2,#3) {\scriptsize #4} ;
}

\def\marked (#1) to (#2) pos=#3 in=#4 out=#5 {
   \draw (#1) to[out=#5,in=#4] node[pos=#3] {$\bullet$} (#2) ;
}

\def\dashedmarked (#1) to (#2) pos=#3 in=#4 out=#5 {
   \draw[dashed] (#1) to[out=#5,in=#4] node[pos=#3] {$\bullet$} (#2) ;
}

\def\leftmarked (#1) to (#2) pos=#3 in=#4 out=#5 w=#6 {
   \draw[line width=2pt] (#1) to[out=#5,in=#4] node[pos=#3] {$\bullet$} node[midway,left] {$#6$} (#2) ;
}

\def\wlmarked (#1) to (#2) pos=#3 in=#4 out=#5 w=#6 {
   \draw (#1) to[out=#5,in=#4] node[pos=#3] {$\bullet$} node[midway,left] {$#6$} (#2) ;
}

\def\rightmarked (#1) to (#2) pos=#3 in=#4 out=#5 w=#6 {
   \draw[line width=2pt] (#1) to[out=#5,in=#4] node[pos=#3] {$\bullet$} node[midway,right] {$#6$} (#2) ;
}

\def\doublemarked (#1) to (#2) pos=#3 in=#4 out=#5 {
   \draw[line width=2pt] (#1) to[out=#5,in=#4] node[pos=#3] {$\bullet$} (#2) ;
}

\newcommand{\R}{\mathbb{R}}

\newcommand{\N}{\mathbb{N}}
\newcommand{\Z}{\mathbb{Z}}
\newcommand{\FF}{\mathbb{F}}

\newcommand{\D}{\mathcal{D}}

\newcommand{\calL}{\mathcal{L}}

\renewcommand{\geq}{\geqslant}
\renewcommand{\leq}{\leqslant}

\newcommand{\<}{\langle}
\renewcommand{\>}{\rangle}

\renewcommand{\bar}{\overline}

\renewcommand{\div}{\mathrm{div}}
\newcommand{\GL}{\mathrm{GL}}
\newcommand{\codeg}{\mathrm{codeg}}

\newcommand{\dsum}{\displaystyle\sum}
\newcommand{\dprod}{\displaystyle\prod}

\renewcommand{\tilde}{\widetilde}

\newcommand{\ang}[1]{\langle #1\rangle}

\newcommand{\LDelta}{\calL_\Delta}
\newcommand{\XDelta}{{X_\Delta}}

\newcommand{\bleft}{b_{\mathrm{left}}}
\newcommand{\bright}{b_{\mathrm{right}}}

\newcommand{\gmax}{g_{\max}}
\newcommand{\smax}{s_{\max}}

\newcommand{\muSDm}{\mu_S(\D,m)}
\newcommand{\muSSDm}{\mu_{S'}(\D,m)}

\newcommand{\muSDDmk}{\mu_S(\D',m'_0)}
\newcommand{\muSDDmkk}{\mu_S(\D',m'_1)}
\newcommand{\muSDDmkkk}{\mu_S(\D',m'_2)}
\newcommand{\muSSDDmk}{\mu_{S'}(\D',m'_0)}
\newcommand{\muSSDDmkk}{\mu_{S'}(\D',m'_1)}
\newcommand{\muSSDDmkkk}{\mu_{S'}(\D',m'_2)}

\newcommand{\muSDmk}{\mu_S(\D,m_0)}
\newcommand{\muSDmkk}{\mu_S(\D,m_1)}
\newcommand{\muSDmkkk}{\mu_S(\D,m_2)}
\newcommand{\muSDmmk}{\mu_S(\D,m'_0)}
\newcommand{\muSDmmkk}{\mu_S(\D,m'_1)}
\newcommand{\muSDmmkkk}{\mu_S(\D,m'_2)}
\newcommand{\muSSDmk}{\mu_{S'}(\D,m_0)}
\newcommand{\muSSDmkk}{\mu_{S'}(\D,m_1)}
\newcommand{\muSSDmkkk}{\mu_{S'}(\D,m_2)}
\newcommand{\muSSDmmk}{\mu_{S'}(\D,m'_0)}
\newcommand{\muSSDmmkk}{\mu_{S'}(\D,m'_1)}
\newcommand{\muSSDmmkkk}{\mu_{S'}(\D,m'_2)}

\newcommand{\ie}{i.e. }

\title{Combinatorial Göttsche-Schroeter invariants in any genus}
\author{Gurvan Mével}
\address{Université de Genève, Section de Mathématiques, rue du Conseil-Général 7-9, 1205 Genève, Suisse}
\email{gurvan.mevel@unige.ch}


\begin{document}

\begin{abstract}
Göttsche-Schroeter invariants are a genus $0$ extension of Block-Göttsche invariants. They interpolate between Welschinger invariants involving pairs of complex conjugated points and genus $0$ descendant Gromov-Witten invariants. They can be computed by a floor diagram algorithm.

In this paper, we show that this floor diagrams recipe actually leads to some invariants in any genus. This generalizes Göttsche-Schroter invariant in higher genus in a combinatorial way. We then prove some polynomiality result and establish a link with invariants defined by Shustin and Sinichkin. 
We provide many examples. In particular, we conjecture that these combinatorial invariants satisfy the Abramovich-Bertram formula.

\end{abstract}

\maketitle

\tableofcontents

\section{Introduction}

\subsection{Enumerative geometry}

Consider $X$ a complex algebraic and non-singular surface, and let $\calL$ be a sufficiently ample line bundle over $X$. We define curves on $X$ as the zero-sets of sections of $\calL$. Given a non-negative integer $\delta$, let $N^\delta(\calL)$ be the number of irreducible curves on $X$ with $\delta$ nodes passing through $\frac{\calL^2 + c_1(X)\cdot\calL}{2}-\delta$ points in generic position. This number is known as a Severi degree. It does not depend on the points configuration as long as it is generic. Because of the adjunction formula, we could consider the dual problem of determining $N_g(\calL)$ the number of curves on $X$ of genus $g$ and passing through $c_1(X)\cdot\calL-1+g$ points. This number corresponds to some Gromov-Witten invariant.

In a real setting, these counts are not invariants as they depend on the configuration of points we choose. However, a genus $0$ real counterpart has been highlighted by Welschinger \cite{welschinger_invariants_2005}. 
He showed that on some surfaces, counting curves passing through a configuration of real points with signs $\pm 1$ leads to an invariant.
More generally, when choosing the configuration of points one can pick $s$ pairs of complex conjugated points. The number of curves passing through the real configuration of points and counted with signs again does not depend on the configuration itself, as long as it has the appropriate number of points and is generic.

It is difficult in general to compute these numbers. It was not before the end of the XXth century that recursive formulas for the complex enumeration have been proven \cite{kontsevich_gromov-witten_1994,caporaso_counting_1998}. The behavior of these counts when the line bundle varies have also been studied, see \cite{di_francesco_quantum_1995, gottsche_conjectural_1998, fomin_labeled_2010, tzeng_proof_2012}  for instance.
Let us last mention that the rational Severi degrees of the Hirzebruch surfaces $\FF_0$ and $\FF_2$ satisfy the Abramovich-Bertram formula \cite{abramovich-2001-formula}. This result has been generalized by Vakil in any genus \cite{vakil-2000-counting}.

\subsection{The tropical approach} \label{sec-trop-enumgeom}

The emergence of tropical geometry provided new ways to compute these numbers. A significant breakthrough is Mikhalkin's correspondence theorem \cite{mikhalkin_enumerative_2005} that turns counts of algebraic curves on toric surfaces into counts of tropical curves with some multiplicities. He also gives a version of his correspondence theorem suitable to determine the Welschinger invariants when $s=0$, \ie when there is no pair of complex conjugated points in the configuration. This has been extended by Shustin \cite{shustin-2006-tropical} to the case $s\geq1$.
Following Mikhalkin's correspondence theorem, Brugallé and Mikhalkin reduced the enumeration of tropical curves to the enumeration of floor diagrams with some multiplicities \cite{brugalle_enumeration_2007, brugalle_floor_2008}.

Through this tropical approach, one can recover some results or prove new ones regarding the enumerative problems we are intested in. For instance, Franz and Markwig gave a tropical proof of the Abramovich-Bertram formula \cite{franz-2011-tropical}. 
Brugallé and Markwig generalized the Abramovich-Bertram and Vakil's formulas to the Hirzebruch surfaces $\FF_n$ and $\FF_{n+2}$, by working in the tropical world and using a correspondence theorem \cite{brugalle-2016-deformation}.

\subsection{Refined invariants}

In the tropical enumeration, Block and Göttsche proposed to use a refined multiplicity, which is no longer an integer but a symmetric Laurent polynomial in a formal variable $q$ \cite{block_refined_2016}. Itenberg and Mikhalkin showed that the count with Block-Göttsche multiplicities also leads to an invariant \cite{itenberg_block-gottsche_2013}, known as the Block-Göttsche invariant and denoted by $G_g(\Delta)(q)$, where $g$ is the genus and $\Delta$ is the polygon which defines the toric surface we look at. Tropical refined invariants have the property to interpolate between complex and real enumeration of curves : plugging $q=1$ we get Gromov-Witten invariant, and plugging $q=-1$ we get tropical Welschinger invariant.

In the rational case, Göttsche and Schroeter extended Block-Göttsche invariants and  defined a refined broccoli invariant now taking into account the number $s$ of pairs of complex conjugated points we fix in the points configuration \cite{gottsche_refined_2019}. These invariants are denoted by $G_0(\Delta,s)(q)$ and correspond to Block-Göttsche invariants for $s=0$. It now interpolates between the broccoli invariants of \cite{gathmann-2013-broccoli}, \ie Welschinger invariants involving pairs of complex conjugated points, and genus $0$ descendant Gromov-Witten invariants.  Göttsche-Schroeter invariants appeared to be a particular case of some invariants defined by Blechman and Shustin \cite{blechman-2019-refined}.
Schroeter and Shustin generalized Göttsche-Schroeter invariants to genus 1 \cite{schroeter-2018-refined}. Simultaneously and independently with this paper, Shustin and Sinichkin proposed a generalization of the work of \cite{schroeter-2018-refined} to any genus \cite{shustin-2024-refined}. They also showed that the evaluation at $q = 1$ gives the number of curves satisfying some incidence and tangency conditions.


The computation of the tropical refined invariants is possible using the floor diagram algorithm, adapted to the refined setting by Block and Göttsche \cite{block_refined_2016}. With an additional decoration called pairing, the floor diagrams can also be used to compute the broccoli invariants in genus $0$ from \cite{gottsche_refined_2019}, see \cite{brugalle_polynomiality_2022}. In particular, the existence of refined invariants ensures that the diagram count does not depend on the chosen pairing. 
Using floor diagrams, Bousseau has shown that Block-Göttsche invariants satisfy the Abramovich-Bertram formula \cite{bousseau_refined_2021}, settling a conjecture of \cite{brugalle-2020-invariance}.

\subsection{Results of this paper}

The calculation of Göttsche-Schroeter invariants $G_0(\Delta,s)$ using floor diagrams requires to choose a pairing $S$ of order $s$, see section \ref{subsec-refined-invariants}. However, the Göttsche-Schroeter invariant \emph{does not} depend on the choice of this pairing, as long as it has order $s$ by \cite[theorem 2.13]{brugalle_polynomiality_2022}, stated as theorem \ref{theo-GS-inv-formula} here. Namely, if $S$ and $S'$ are two pairings of order $s$, one can define the count of floor diagrams $G_0(\Delta,S)$ and $G_0(\Delta,S')$, and show they are both equal to the Göttsche-Schroeter invariant $G_0(\Delta,s)$ (this last notation is then an abuse of notation).

We give in this paper a combinatorial proof of this independence which is valid in any genus, not only in the rational case. For any genus $g$ we define a quantity $G_g(\Delta,S)$ as a count of floor diagrams, and show it does not depend on $S$ but only of its order.

\begin{named}{Theorem \ref{theo-invS}}
Let $\Delta$ be $h$-transverse polygon and $g\in\N$. Let $s \in \N$ and $S,S'$ be two pairings of order $s$. Then $G_g(\Delta,S) = G_g(\Delta,S')$. We can then write $G_g(\Delta,s)$ and call it \emph{Göttsche-Schroeter invariant of genus $g$}.
\end{named}

As wished is \cite[remark 2.14]{brugalle_polynomiality_2022} the proof is entirely combinatorial and does not go through tropical geometry.  
Moreover, in the case where the polygon $\Delta$ is $h$-transverse we show that the invariants of \cite{shustin-2024-refined}, denoted by $RB_q(\Delta,g,(n_1,n_2))$, match the ones of this paper. Hence, the floor diagrams algorithm gives a practical way to study and compute the invariants of \cite{shustin-2024-refined}.

\begin{named}{Proposition \ref{prop-equal-shustin}}
Let $\Delta$ be a $h$-transverse polygon, $g\in\N$ and $s\in\N$. The combinatorial Göttsche-Schroeter invariant corresponds to the invariant of \cite{shustin-2024-refined}, \ie
\[ G_g(\Delta,s)(q) = RB_q(\Delta,g,(y(\Delta)-1+g-2s,s)) . \]
\end{named}

We then illustrate the use of floor diagrams by proving few results on this higher genus Göttsche-Schroeter invariant $G_g(\Delta,s)$. These properties extend the ones we can find in \cite{brugalle_polynomiality_2022}. Especially we show some polynomiality behavior with respect to $s$, which generalizes \cite[theorem 1.7]{brugalle_polynomiality_2022} to arbitrary genus. Here, $\< G_g(\Delta,s) \>_i$ denotes the codegree $i$ coefficient of $G_g(\Delta,s)$. Other notations are defined in sections \ref{subsec-polygons-floor-diagrams} and \ref{subsec-refined-invariants}.

\begin{named}{Theorem \ref{theo-GSinv-poly-s}}
Let $\Delta$ be a $h$-transverse polygon and $g \leq \gmax(\Delta)$. If $2i \leq e^{-\infty}(\Delta)$ and $i \leq \gmax(\Delta)$, then the values $\< G_g(\Delta,s) \>_i$ for $0 \leq s \leq \smax(\Delta,g)$ are interpolated by a polynomial of degree $i$, whose leading coefficient is $\frac{(-2)^i}{i!} \binom{g_{\max}-i}{g}$.
\end{named}

We also perform computations on manageable examples. This leads to few conjectures that may give evidence that this combinatorial invariant may have a geometric interpretation. In particular the higher genus Göttsche-Schroeter invariants seem to satisfy the Abramovich-Bertam formula. Here, the polygon $\Delta^n_{a,b}$ defines the Hirzebruch surface $\FF_n$ together with the curves of bidegree $(a,b)$, see figure \ref{fig-trapeze-Hirz}.

\begin{named}{Conjecture \ref{conj-AB}, Abramovich-Bertram formula}
    Let $a,b \in\N$ and $g\geq 0$. For any $s\geq0$ one has
    \[ G_g(\Delta^0_{a,a+b},s) = \dsum_{j=0}^a \binom{b+2j}{j} G_g(\Delta^2_{a-j,b+2j},s).  \]
\end{named}

\addtocontents{toc}{\SkipTocEntry}
\subsection*{Acknowledgments}

I would like to thank Erwan Brugallé for asking me this question and many others three years ago, for his constant support since then, and for his help with some parts of this paper.
I also thank Eugenii Shustin for drawing my attention to a paper I was unaware of. 

This work was mainly conducted within the France 2030 framework programme, Centre Henri Lebesgue ANR-11-LABX-0020-01.
I acknowledge financial support from the CNRS and from the Swiss National Science Foundation grant 204125.

\section{Floor diagrams and refined invariants in genus $0$} \label{subsec-polydiag}

In this section we recall how to use floor diagrams to genus $0$ compute refined invariants.

\subsection{$h$-transverse polygons and floor diagrams} \label{subsec-polygons-floor-diagrams}

We first introduce some definitions and notations.
In this paper a polygon will always be a convex polygon in $\R^2$ with vertices in $\Z^2$.

\begin{defi} \label{def-htransverse}
Let $\Delta$ be a polygon. We said that $\Delta$ is \emph{$h$-transverse} if any of its edge has an outward normal vector of the form $(0,\pm 1)$ or $(\pm 1,n)$ for some $n\in\Z$.
\end{defi}

Via toric geometry, a polygon $\Delta$ defines a toric surface $\XDelta$ and a line bundle $\LDelta$ on $\XDelta$. 
It also has some combinatorial data that is related to the enumerative problems we are interested in throughout this text. We set the following notations :
\begin{itemize}
	\item $a(\Delta)$ is the height of $\Delta$, \ie the difference between the maximal and the minimal ordinate of a point of $\Delta$,
	
	\item $e^{+\infty}(\Delta)$ (resp. $e^{-\infty}(\Delta)$) is the length of the top (resp. bottom) horizontal edge of $\Delta$.
 
  	\item $y(\Delta) = |\partial \Delta \cap \Z^2|$ the number of integer points on the boundary of $\Delta$, geometrically it is equal to $- \calL_\Delta \cdot K_{X_\Delta}$, 

  	\item $\chi(\Delta)$ is the number of vertices of $\Delta$, geometrically it is the Euler characteristic of $X_\Delta$,

   \item $\gmax(\Delta) = |\mathring\Delta \cap\Z^2|$ the number of interior lattice points of $\Delta$, geometrically it is the maximal genus of a curve in the linear system associated to $\LDelta$ if $X_\Delta$ is non-singular,
  	
  	\item $\smax(\Delta,g) = \left\lfloor \dfrac{y(\Delta)-1+g}{2} \right \rfloor$ for $g\in\N$.
\end{itemize}
Note that $y(\Delta) = e^{+\infty}(\Delta) + e^{-\infty}(\Delta) + 2a(\Delta)$.
Moreover, if $\Delta$ is $h$-transverse we denote :
\begin{itemize}
	\item 
    $\bleft(\Delta)$ (resp. $\bright(\Delta)$) is the unordered list of integers $k$ appearing $j$ times, where $j$ is the integral length of the side of $\Delta$ having $(-1,k)$ (resp. $(1,k)$) as outward normal vector. 
\end{itemize}
When no ambiguity is possible we will simply use $a$, $e^{+\infty}$, $\gmax$, etc.

\begin{ex}
\begin{figure}[h!]
	\begin{subfigure}[t]{0.33\textwidth}
	\centering
	\begin{tikzpicture}[scale=2/3] 
	
	\foreach \p in {(0,0), (1,0), (2,0), (3,0), (0,1), (1,1), (2,1), (0,2), (1,2), (0,3)}
			{\draw node at \p {$\bullet$} ;}
	\draw (0,0) to (3,0) to (0,3) to cycle;
		
	\end{tikzpicture}
	\caption{$\Delta_a$.}
	\label{fig-ex-poly-d3}
\end{subfigure}
\begin{subfigure}[t]{0.33\textwidth}
	\centering
	\begin{tikzpicture}[scale=2/3] 
	
	\foreach \p in {(0,0), (1,0), (2,0),
                    (0,1), (1,1), (2,1), (3,1), (4,1),
                    (0,2), (1,2), (2,2), (3,2), (4,2),
                    (1,3), (2,3), (3,3), (4,3),
                    (2,4), (3,4)}
			{\draw node at \p {$\bullet$} ;}
	\draw (0,0) to (2,0) to (4,1) to (4,3) to (3,4) to (2,4) to (0,2) to cycle;
		
	\end{tikzpicture}
	\caption{$\Delta_b$.}
	\label{fig-expolyb}
\end{subfigure}
\begin{subfigure}[t]{0.32\textwidth}
	\centering
	\begin{tikzpicture}[scale=2/3] 
	
	\foreach \p in {(0,0), (1,0), (2,0), (1,1), (2,1), (3,1), (1,2), (2,2), (3,2), (1,3)}
			{\draw node at \p {$\bullet$} ;}
	\draw (0,0) to (2,0) to (3,1) to (3,2) to (1,3) to cycle;
		
	\end{tikzpicture}
	\caption{$\Delta_c$.}
	\label{fig-expolyc}
\end{subfigure}
	\caption{Some polygons.}
	\label{fig-expoly}
\end{figure}
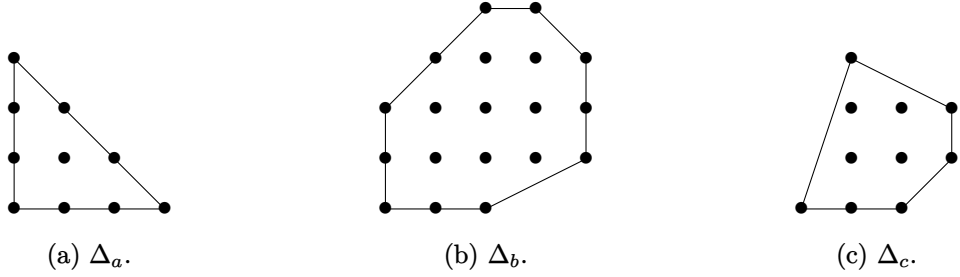

Consider the polygons of figure \ref{fig-expoly}. The polygons $\Delta_a$ and $\Delta_b$ are $h$-transverse but $\Delta_c$ is not. We give in table \ref{table-ex-polygons} their combinatorial data.

\begin{table}[h!]
\renewcommand{\arraystretch}{1.5}
\centering
\begin{tabular}{|c||c|c|c|c|c|c|c|c|}
\hline 
 & $a(\Delta)$ & $e^{+\infty}(\Delta)$ & $e^{-\infty}(\Delta)$ & $y(\Delta)$ & $\chi(\Delta)$ & $\gmax(\Delta)$ & $\bleft(\Delta)$ & $\bright(\Delta)$ \\
 \hline
 \hline
$\ \Delta_a\ $ & $3$ & $0$ & $3$ & $9$ & $3$ & $1$& $\{ 0,0,0 \}$ & $\{ 1,1,1 \}$  \\
\hline 
$\ \Delta_b\ $ & $4$ & $1$ & $2$ & $11$ & $7$ & $8$ & $\{ 0,0,1,1 \}$ & $\{ -2,0,0,1 \}$ \\
\hline
$\ \Delta_c\ $ & $3$ & $0$ & $2$ & $6$ & $5$ & $4$ & $\not$ & $\not$ \\
\hline 
\end{tabular}
	\caption{Combinatorial data of the polygons of figure \ref{fig-expoly}.}
\label{table-ex-polygons}
\end{table}
\end{ex}

We now introduce some terminology on graphs.
An \textit{oriented graph} $\Gamma$ is a collection of vertices $V(\Gamma)$, of bounded edges $E^0(\Gamma)$ (\ie of oriented edges adjacent to two vertices), and of infinite edges $E^{\infty}(\Gamma)$ (\ie of oriented edges adjacent to one vertex).
An infinite edge oriented toward (resp. from) its adjacent vertex is called a \emph{source} (resp. a \emph{sink}), and we denote by $E^{-\infty}(\Delta)$ the set of sources (resp. by $E^{+\infty}(\Gamma)$ the set of sinks). We denote by $E(\Gamma)$ the set of all edges of $\Gamma$. 
The graph $\Gamma$ is \textit{weighted} if there is a function $w : E(\Gamma) \to \N^*$. Given a vertex $v\in V(\Gamma)$ of an oriented weighted graph, its \textit{divergence} $\div(v)$ is the difference of the weights entering and leaving $v$, \ie
\[ \div(v) = \dsum_{ \overset{e}{\to} v} w(e) - \dsum_{v \overset{e}{\to}} w(e).\]
Last, the \emph{genus} of a graph $\Gamma$ is its first Betti number.

\begin{defi}[Floor diagram] \label{def-floordiag}
Let $\Delta$ be a $h$-transverse polygon and $g\in\N$. A \emph{floor diagram} $\D$ with Newton polygon $\Delta$ and genus $g$ is a quadruple $(\Gamma,w,L,R)$ such that 
\begin{itemize}
	\item $(\Gamma,w)$ is a weighted, connected, oriented and acyclic graph of genus $g$,
 
    \item the graph $\Gamma$ has $a(\Delta)$ vertices, $e^{+\infty}(\Delta)$ sinks and $e^{-\infty}(\Delta)$ sources,
	
	\item all the infinite edges have weight 1,
	
	\item $L : V(\Gamma) \to \bleft(\Delta)$ and $R : V(\Gamma) \to \bright(\Delta)$ are bijections such that for every vertex $v \in V(\Gamma)$ one has $\div(v) = L(v) + R(v)$. 
\end{itemize}
\end{defi}

By abuse of notations, we will use $\D$ for $\Gamma$. If $\D$ is a floor diagram its number of elements $n(\D)$ is its number of vertices and edges, \ie 
\[ n(\D) = |V(\D)| + |E(\D) |. \]
Since one has $|E(\D)| = |E^0(\D)| + |E^\infty(\D)|$, $|V(\D)| - |E^0(\D)| = 1-g$ with $g$ the genus of $\D$, and $|V(\D)| = a(\Delta)$ with $\Delta$ the Newton polygon of $\D$, then 
\[ n(\D) = y(\Delta)-1+g . \]
The \emph{degree} of $\D$ is 
    \[ \deg(\D) = \dsum_{e\in E(\D)} (w(e)-1) . \]
If the diagram $\D$ has Newton polygon $\Delta$ and genus $g$, its \emph{codegree} is
    \[ \codeg(\D) = \gmax(\Delta) -g-\deg(\D) . \]

We will always draw the floor diagrams oriented from bottom to top. Hence we do not put any arrow on the edges to show the orientation. Moreover we indicate the weights of the edges only if their are at least $2$. We give some examples of floor diagrams.

\begin{ex}
Figure \ref{fig-ex-diag-d3} gives all the floor diagrams with Newton polygon the polygon of figure \ref{fig-ex-poly-d3}. Here, the functions $R$ and $L$ are constant equal to $1$ and $0$, so any vertex has divergence $1$. The first three diagrams have genus 0, and the last one has genus 1. We also precise their codegrees.

\begin{figure}[h!]
	\begin{subfigure}[t]{0.24\textwidth}
	\centering
	\begin{tikzpicture}[scale=5/6]
		\floor (3) at (0,4.5) {} ;
		\floor (2) at (0,3) {} ;
		\floor (1) at (-0.3,1.5) {} ;
		
		\draw (3) to (2) ;
		\draw (2) to[out=-115, in=90] (1) ;
		
		\draw (1) to[out=-115, in=90] (-0.6,0);
		\draw (1) to[out=-65, in=90] (0,0);
		\draw (2) to[out=-65, in=90] (0.6,0);
	\end{tikzpicture}
	\caption{$g=0$, \\ $\codeg(\D) = 1$}
    \label{fig-ex-diag-d3-a}
\end{subfigure}
\begin{subfigure}[t]{0.24\textwidth}
	\centering
	\begin{tikzpicture}[scale=5/6]
		\floor (3) at (0,4.5) {} ;
		\floor (2) at (0,3) {} ;
		\floor (1) at (0,1.5) {} ;
		
		\draw (3) to (2) ;
		\draw (2) to node[left] {\scriptsize $2$} (1) ;		
		
		\draw (1) to[out=-130, in=90] (-0.5,0);
		\draw (1) to[out=-90, in=90] (0,0) ;
		\draw (1) to[out=-50, in=90] (0.5,0);
		
	\end{tikzpicture}
	\caption{$g=0$, \\ $\codeg(\D) = 0$}
    \label{fig-ex-diag-d3-b}
\end{subfigure}
\begin{subfigure}[t]{0.24\textwidth}
	\centering
	\begin{tikzpicture}[scale=5/6]
		\floor (3) at (1,3) {} ;
		\floor (2) at (-1,3) {} ;
		\floor (1) at (0,1.5) {} ;
		
		\draw (3) to (1) ;
		\draw (2) to (1) ;
		
		\draw (1) to[out=-130, in=90] (-0.5,0);
		\draw (1) to[out=-90, in=90] (0,0) ;
		\draw (1) to[out=-50, in=90] (0.5,0);
	\end{tikzpicture}
	\caption{$g=0$, \\ $\codeg(\D) = 1$}
    \label{fig-ex-diag-d3-c}
\end{subfigure}
\begin{subfigure}[t]{0.24\textwidth}
	\centering
	\begin{tikzpicture}[scale=5/6]
		\floor (3) at (0,4.5) {} ;
		\floor (2) at (0,3) {} ;
		\floor (1) at (0,1.5) {} ;
		
		\draw (3) to (2) ;
		\draw (2) to[out=-60, in=60] (1) ;	
		\draw (2) to[out=-120, in=120] (1) ;		
		
		\draw (1) to[out=-130, in=90] (-0.5,0);
		\draw (1) to[out=-90, in=90] (0,0) ;
		\draw (1) to[out=-50, in=90] (0.5,0);
	\end{tikzpicture}
	\caption{$g=1$, \\ $\codeg(\D) = 0$}
    \label{fig-ex-diag-d3-d}
\end{subfigure}
	\caption{The floor diagrams with Newton polygon the polygon of figure \ref{fig-ex-poly-d3}.}
	\label{fig-ex-diag-d3}
\end{figure}
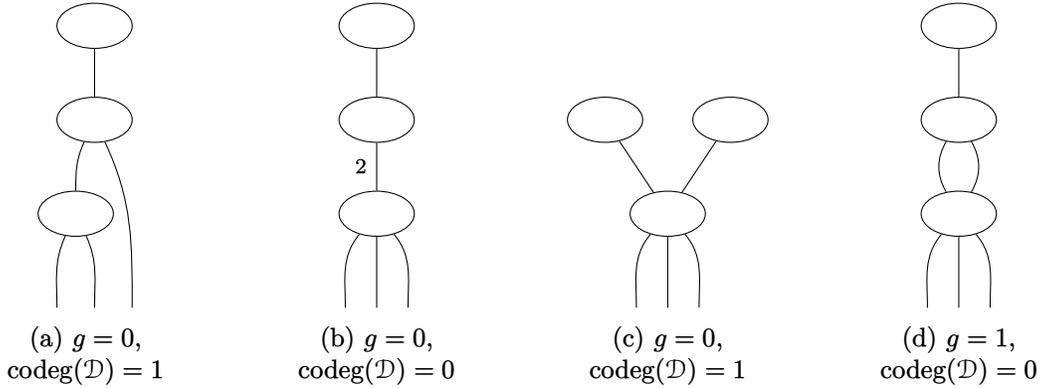
\end{ex}

\subsection{Refined invariants} \label{subsec-refined-invariants}

Following \cite{brugalle_polynomiality_2022}, we now recall how to determine the Göttsche-Schroeter invariants of \cite{gottsche_refined_2019} using floor diagrams.

The orientation of a floor diagram $\D$ induces a partial order $\prec$ on the set of its elements $E(\D) \cup V(\D)$. More precisely, given two elements $\alpha$ and $\beta$ we write $\alpha \prec \beta$ if there exists an oriented path in $\D$ from $\alpha$ to $\beta$.  Hence, one can define increasing functions on a floor diagram.

\begin{defi}[Marking]
Let $\D$ be a floor diagram. A \emph{marking} of $\D$ is an increasing bijection 
\[ m : E(\D) \cup V(\D) \to \{1,\dots,n(\D) \} \]
The couple $(\D,m)$ is called a \emph{marked floor diagram}.
 
 Two marked floor diagrams $(\D,m)$ and $(\D',m')$ are \emph{isomorphic} if there exists an isomorphism $\varphi : \D \to \D'$ of weighted graphs such that $L = L' \circ \varphi$, $R = R'\circ \varphi$ and $m=m' \circ \varphi$.
 
 We denote by $\nu(\D)$ the number of markings of a diagram $\D$ up to isomorphisms.
\end{defi}

\begin{ex}
Figure \ref{fig-exdiagmark} gives examples of markings of the floor diagram of figure \ref{fig-ex-diag-d3-a}. The marked floor diagrams of figures \ref{fig-exdiagmarka} and \ref{fig-exdiagmarkb} are isomorphic.

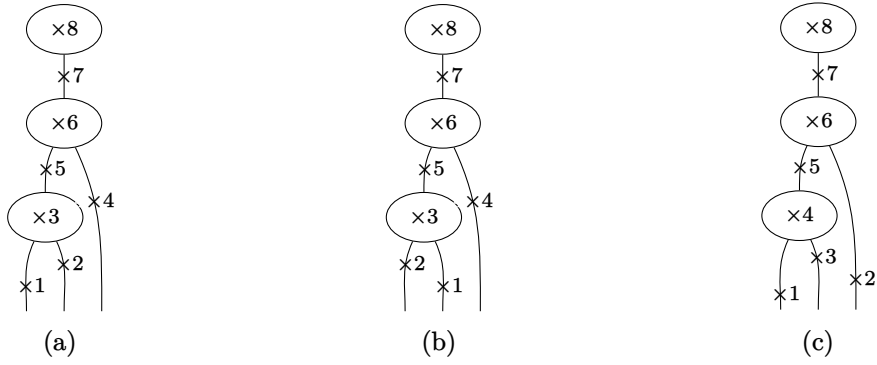
\begin{figure}[h!]
	\begin{subfigure}[t]{0.33\textwidth}
	\centering
	\begin{tikzpicture}[scale=5/6]
		\floor (8) at (0,4.5) {\scriptsize $\times8$};
		\floor (6) at (0,3) {\scriptsize $\times6$};
		\floor (3) at (-0.3,1.5) {\scriptsize $\times3$};
		\draw (8) to node[pos=0.5]{\scriptsize \textcolor{white}{7}$\times7$} (6) ;
		\draw (6) to[out=-115, in=90] node[pos=0.5]{\scriptsize \textcolor{white}{5}$\times5$} (3) ;
		\draw (3) to[out=-115, in=90] node[pos=2/3]{\scriptsize \textcolor{white}{1}$\times1$} (-0.6,0);
		\draw (3) to[out=-65, in=90] node[pos=1/3]{\scriptsize \textcolor{white}{2}$\times2$} (0,0);
		\draw (6) to[out=-65, in=90] node[pos=1/3]{\scriptsize \textcolor{white}{4}$\times4$} (0.6,0);
	\end{tikzpicture}
	\caption{}
	\label{fig-exdiagmarka}
\end{subfigure}
\begin{subfigure}[t]{0.3\textwidth}
	\centering
	\begin{tikzpicture}[scale=5/6]
		\floor (8) at (0,4.5) {\scriptsize $\times8$};
		\floor (6) at (0,3) {\scriptsize $\times6$};
		\floor (3) at (-0.3,1.5) {\scriptsize $\times3$};
		\draw (8) to node[pos=0.5]{\scriptsize \textcolor{white}{7}$\times7$} (6) ;
		\draw (6) to[out=-115, in=90] node[pos=0.5]{\scriptsize \textcolor{white}{5}$\times5$} (3) ;
		\draw (3) to[out=-115, in=90] node[pos=1/3]{\scriptsize \textcolor{white}{2}$\times2$} (-0.6,0);
		\draw (3) to[out=-65, in=90] node[pos=2/3]{\scriptsize \textcolor{white}{1}$\times1$} (0,0);
		\draw (6) to[out=-65, in=90] node[pos=1/3]{\scriptsize \textcolor{white}{4}$\times4$} (0.6,0);
	\end{tikzpicture}
	\caption{}
	\label{fig-exdiagmarkb}
\end{subfigure}	
\begin{subfigure}[t]{0.33\textwidth}
	\centering
	\begin{tikzpicture}[scale=5/6]
		\floor (8) at (4,4.5) {\scriptsize $\times8$};
		\floor (6) at (4,3) {\scriptsize $\times6$};
		\floor (3) at (3.7,1.5) {\scriptsize $\times4$};
		\draw (8) to node[pos=0.5]{\scriptsize \textcolor{white}{7}$\times7$} (6) ;
		\draw (6) to[out=-115, in=90] node[pos=0.5]{\scriptsize \textcolor{white}{5}$\times5$} (3) ;
		\draw (3) to[out=-115, in=90] node[pos=4/5]{\scriptsize \textcolor{white}{1}$\times1$} (3.4,0);
		\draw (3) to[out=-65, in=90] node[pos=1/4]{\scriptsize \textcolor{white}{3}$\times3$} (4,0);
		\draw (6) to[out=-65, in=90] node[pos=5/6]{\scriptsize \textcolor{white}{2}$\times2$} (4.6,0);
	\end{tikzpicture}
	\caption{}
	\label{fig-exdiagmarkc}
\end{subfigure}
	\caption{Some marked floor diagrams with Newton polygon the polygon of figure \ref{fig-ex-poly-d3}.}
	\label{fig-exdiagmark}
\end{figure}
\end{ex}

 A \emph{pairing of order $s$} of the set $\{1,\dots,n\}$ is a set $S$ of $s$ disjoint pairs $\{i,i+1\} \subset \{1,\dots,n\}$. 
 Given a floor diagram $\D$ and a pairing $S$ of $\{1,\dots,n(\D)\}$, we say that a marking $m$ is \emph{compatible} with $S$ if for any $\alpha \in S$, the set $m^{-1}(\alpha)$ consists of
 \begin{itemize}
 	\item either an edge and an adjacent vertex,
 	\item or two edges that are both entering or both leaving the same vertex.
 \end{itemize}
Let $(\D,m)$ be a marked floor diagram and $S$ a pairing compatible with $m$. We define
\begin{align*}
E_0 &= \{e \in E(\D)\ |\ \forall \alpha\in S, e \notin m^{-1}(\alpha) \}, \\
E_1 &= \{e \in E(\D)\ |\ \exists v\in V(\D), \exists \alpha\in S, \{e,v\} = m^{-1}(\alpha) \} , \\
E_2 &= \{\{e,e'\} \subset E(\D)\ |\ \exists \alpha\in S, \{e,e'\} = m^{-1}(\alpha) \}. \\
\end{align*}

For $n\in\Z$ the quantum integer $[n](q)$ is defined by
\[ [n](q) = \dfrac{q^{n/2}-q^{-n/2}}{q^{1/2}-q^{-1/2}} = q^{(n-1)/2} + q^{(n-3)/2} + \cdots + q^{-(n-3)/2} + q^{-(n-1)/2} \in \N[q^{\pm 1/2}] . \]
We will use the shorcuts
\[ [n]^2 = [n](q)^2 \text{ and } [n]_2 = [n](q^2).  \]

\begin{defi}[Refined $S$-multiplicity] \label{def-mult-GS}
The \emph{refined $S$-multiplicity} of a marked floor diagram $(\D,m)$ is
\[ \mu_S(\D,m)(q) = \dprod_{e\in E_0} [w(e)]^2 \dprod_{e\in E_1} [w(e)]_2 \dprod_{\{e,e'\} \in E_2} \dfrac{[w(e)][w(e')][w(e)+w(e')]}{[2]} \in \Z[q^{\pm 1/2}] \]
if $S$ and $m$ are compatible, and $\mu_S(\D,m)(q)=0$ otherwise. If non-zero, it is a Laurent polynomial of degree $\deg(\D)$.
\end{defi}

The following theorem can be taken as a definition of the Göttsche-Schroeter invariants.

\begin{theo}[{\cite[theorem 2.13]{brugalle_polynomiality_2022}}]\label{theo-GS-inv-formula}
Let $\Delta$ be a $h$-transverse polygon and $s\in \{0,\dots,\smax(\Delta,0)\}$. For any pairing $S$ of order $s$ of $\{1,\dots, y(\Delta)-1 \}$ one has
\[ G_0(\Delta,s) = \dsum_{(\D,m)} \mu_S(\D,m)  \]
where the sum runs over the isomorphism classes of marked floor diagrams with Newton polygon $\Delta$ and genus $0$.
\end{theo}

\begin{rk}
The theorem implies that the right-hand side does not depend on the pairing $S$ as long as it has order $s$. Thus, to study $G_0(\Delta,s)$ we can choose a particular pairing which makes the calculations easier.

This paper is mainly devoted to prove that we can define an analogous combinatorial quantity for any genus, see theorem \ref{theo-invS}. More precisely, we will give a combinatorial proof that, in any genus $g$, the sum of the right hand side of theorem \ref{theo-GS-inv-formula} does not depend on $S$, leading to a quantity we will denote $G_g(\Delta,s)$.
\end{rk}

\begin{ex} \label{ex-calcul-G0-Delta3}
    Let $\D_1$, $\D_2$ and $\D_3$ be the diagrams of figures \ref{fig-ex-diag-d3-a}, \ref{fig-ex-diag-d3-b} and \ref{fig-ex-diag-d3-c}. The following table \ref{table-ex-refined-inv} gives their contributions to the Göttsche-Schroeter invariant, using the pairing $S = \{\{1,2\},\dots,\{2s-1,2s\}\}$ of order $s$. Hence one has $G_0(\Delta_a,s) = q+(10-2s)+q^{-1}$.

\begin{table}[h!]
\renewcommand{\arraystretch}{1.5}
\centering
\begin{tabular}{|c||c|c|c|c|c|c|c|c|}
\hline 
 & $s=0$ & $s=1$ & $s=2$ & $s=3$ & $s=4$ \\
 \hline
 \hline
$\ \D_1\ $ & $5$ & $3$ & $1$ & $1$ & $1$ \\
\hline 
$\ \D_2\ $ & $q+2+q^{-1}$ & $q+2+q^{-1}$ & $q+2+q^{-1}$ & $q+q^{-1}$ & $q+q^{-1}$  \\
\hline
$\ \D_3\ $ & $3$ & $3$ & $3$ & $3$ & $1$ \\
\hline 
 \hline
$\ G_0(\Delta_a,s) \ $ & $q+10+q^{-1}$ & $q+8+q^{-1}$ & $q+6+q^{-1}$ & $q+4+q^{-1}$ & $q+2+q^{-1}$ \\
\hline 
\end{tabular}
\caption{Computation of $G_0(\Delta_a,s)$.}
\label{table-ex-refined-inv}
\end{table}
\end{ex}

\begin{rk} 
A \emph{lattice preserving transformation} is an application $f : \R^2 \to \R^2$ obtained as a composition of
\begin{itemize}
    \item isomorphisms of $\R^2$ induced by elements of $\GL_2(\Z)$,
    \item translations that preserves the lattice $\Z^2$, \ie translations by a vector $u\in\Z^2$.
\end{itemize} 
In other words, a lattice preserving transformation is an element of the affine group of $\R^2$ for which the lattice $\Z^2$ is invariant.  
We say that  $\Delta$ and $\Delta'$ are \emph{congruent} if there exists a lattice preserving transformation $f$ such that $\Delta' = f(\Delta)$. If $\Delta$ and $\Delta'$ are congruent, then $G_0(\Delta,s) = G_0(\Delta',s)$.
Indeed, a translation does not change the family of floor diagrams defined by $\Delta$. Moreover, a floor diagram is a way to encode a tropical curve $C$. Via the dual subdivision of $\Delta$ corresponding to $C$, a matrix of $\GL_2(\Z)$ which acts on $\Delta$ also acts on $C$, and preserves its multiplicity. Hence the total count does not change.
\end{rk}

\subsection{Operations on floor diagrams}\label{subsec-lemmas-FD}

We will use the following operations on floor diagrams, introduced in \cite{brugalle_polynomiality_2022}.

\begin{figure}[h!]
	\begin{subfigure}[t]{0.49\textwidth}
	\centering
	\begin{tikzpicture}[scale=1]
		\floor (2) at (1,3) {$v_2$} ;
		\floor (1) at (0,1.5) {$v_1$} ;
		
		\draw (1) to node[right,pos=1/2]{\scriptsize $w(e_1)$} (2) ;
		\draw (1) to[out=100, in=-90] node[right,pos=4/5]{\footnotesize $w(e_2)$}  (-0.2,4.5);
		
		\draw[->] (2.2,3) to (3.3,3) ;
		
		\floor (4) at (5,3) {$v_2$} ;
		\floor (3) at (4,1.5) {$v_1$} ;
		
		\draw (3) to node[right,pos=1/2]{\scriptsize $w(e_1)+w(e_2)$} (4) ;
		\draw (4) to[out=90, in=-90] node[left,pos=1/2]{\scriptsize $w(e_2)$}  (5,4.5);
	\end{tikzpicture}
	\caption{Operation $A^+$.}
	\label{fig-operationsA+}
\end{subfigure}
\begin{subfigure}[t]{0.49\textwidth}
	\centering
	\begin{tikzpicture}[scale=1]
		\floor (2) at (0,3) {$v_2$} ;
		\floor (1) at (1,1.5) {$v_1$} ;
		
		\draw (1) to node[right,pos=1/2]{\scriptsize $w(e_1)$} (2) ;
		\draw (2) to[out=-100, in=90] node[right,pos=4/5]{\footnotesize $w(e_2)$}  (-0.2,0);
		
		\draw[->] (2.2,1.5) to (3.3,1.5) ;
		
		\floor (4) at (4,3) {$v_2$} ;
		\floor (3) at (5,1.5) {$v_1$} ;
		
		\draw (3) to node[right,pos=1/2]{\scriptsize $w(e_1)+w(e_2)$} (4) ;
		\draw (3) to[out=-90, in=90] node[left,pos=1/2]{\scriptsize $w(e_2)$}  (5,0);	
	\end{tikzpicture}
	\caption{Operation $A^-$.}
	\label{fig-operationsA-}
\end{subfigure}
	\caption{Operations $A^+$ and $A^-$.}
	\label{fig-operationsA}
\end{figure}
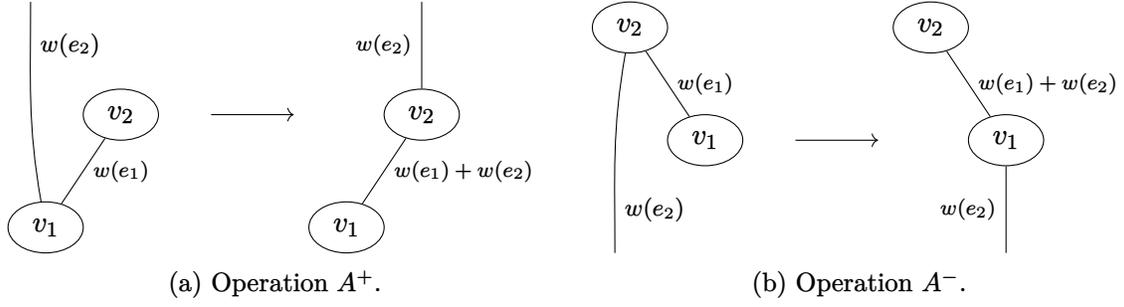

\begin{enumerate}
	\item[$A^+$ : ] If there are vertices $v_1 \prec v_2$ connected by an edge $e_1$ and another edge $e_2$ leaving $v_1$ but not entering $v_2$, then we construct a new diagram as depicted in figure \ref{fig-operationsA+}.
	
	\item[$A^-$ : ] Similarly if $e_2$ is entering $v_2$ but not leaving $v_1$, see figure \ref{fig-operationsA-}.
\end{enumerate}

\begin{lem}[{\cite[lemma 3.2]{brugalle_polynomiality_2022}}]\label{lem-opeAB}
Genus and Newton polygon are invariant under operations $A^\pm$. Moreover, the codegree decreases by $w(e_2)$ under operations $A^\pm$.
\end{lem}

\section{Refined invariants in the non-rational case}

\subsection{Definition of $G_g(\Delta,s)$}

\begin{defi} \label{def-GDeltaS}
Let $\Delta$ be a $h$-transverse polygon, $g\in\N$, $s\in \N$ and $S$ be a pairing  of order $s$ of $\{1,\dots, 2\smax(\Delta,g) \}$. We define 
\[ G_g(\Delta,S) = \dsum_{(\D,m)} \mu_S(\D,m) \in \Z[q^{\pm1}]  \]
where the sum runs over the isomorphism classes of marked floor diagrams with Newton polygon $\Delta$ and genus $g$.
\end{defi}

The goal is now to turn the $S$-dependence into a $s$-dependence. We start with a technical lemma on quantum integers. Remember we denote $[n]^2 = [n](q)^2$ and $[n]_2 = [n](q^2)$. 

\begin{lem}\label{lem-quantum} 
Let $a,b \in\Z$ be integers. Then 
\begin{align*}
2[a][b][a+b] &= [2]\left( [a+b]^2 [a]_2 - [a+b]_2 [a]^2 \right) \\
 &= [2] \left([a]^2[b]_2 + [a]_2 [b]^2 \right).
\end{align*}
\end{lem}

\begin{demo}
%
The first quantity is
\begin{align*}
2[a][b][a+b] &= 2 \dfrac{(q^{a/2}-q^{-a/2}) (q^{b/2}-q^{-b/2}) (q^{(a+b)/2}-q^{-(a+b)/2})}{(q^{1/2}-q^{-1/2})^3}  \\
  &= 2 \dfrac{ q^{a+b} - q^{-a-b} -q^a + q^{-a} -q^b +q^{-b} }{(q^{1/2}-q^{-1/2})^3}.
\end{align*}

To show the equalities, for any integers $c,d$ we first compute
\begin{align*}
[2][c]^2[d]_2 &= \dfrac{q-q^{-1}}{q^{1/2}-q^{-1/2}} \times \left(\dfrac{q^{c/2}-q^{-c/2}}{q^{1/2}-q^{-1/2}}\right)^2 \times \dfrac{q^c-q^{-c}}{q-q^{-1}}  \\
  &= \dfrac{q^{c+d} - q^{c-d} - 2(q^d-q^{-d}) + q^{-c+d} - q^{-c-d}}{(q^{1/2}-q^{-1/2})^3}.
\end{align*}
Applying this to $(c,d) = (a+b,a)$ and $(c,d) = (a,a+b)$ we deduce that
\begin{align*}
 [2]\left([a+b]^2[a]_2 - [a+b]_2[a]^2\right) 
    ={} &   \dfrac{q^{2a+b} - q^b- 2(q^a-q^{-a}) + q^{-b} - q^{-2a-b}}{(q^{1/2}-q^{-1/2})^3} \\
    & - \dfrac{q^{2a+b} - q^{-b} - 2(q^{a+b}-q^{-a-b}) + q^b - q^{-2a-b}}{(q^{1/2}-q^{-1/2})^3} \\
    ={} & 2 \dfrac{ q^{a+b} - q^{-a-b} -q^a + q^{-a} -q^b +q^{-b} }{(q^{1/2}-q^{-1/2})^3} \\
    ={} & 2[a][b][a+b],
\end{align*}
and applying it to $(c,d) = (a,b)$ and $(c,d) = (b,a)$ we get
\begin{align*}
[2]\left([a]^2[b]_2 + [a]_2[b]^2\right) 
    ={} &  \dfrac{q^{a+b} - q^{a-b} - 2(q^b-q^{-b}) + q^{-a+b} - q^{-a-b}}{(q^{1/2}-q^{-1/2})^3} \\
    & + \dfrac{q^{a+b} - q^{-a+b} - 2(q^a-q^{-a}) + q^{a-b} - q^{-a-b}}{(q^{1/2}-q^{-1/2})^3} \\
    ={} & 2 \dfrac{ q^{a+b} - q^{-a-b} -q^a + q^{-a} -q^b +q^{-b} }{(q^{1/2}-q^{-1/2})^3} \\
    ={} & 2[a][b][a+b] 
\end{align*}
so the three quantities are equal.
\end{demo}

We can now prove the main result of this paper.

\begin{theo} \label{theo-invS}
Let $\Delta$ be $h$-transverse polygon and $g\in\N$. Let $s \in \N$ and $S,S'$ be two pairings of order $s$. Then $G_g(\Delta,S) = G_g(\Delta,S')$.
\end{theo}

The strategy to prove the theorem is the following. We will determine a partition $(P_k)_k$ of the marked floor diagrams such that for any $k$ one has
\[ \sum_{(\D,m) \in P_k} \mu_S(\D,m) = \sum_{(\D,m) \in P_k} \mu_{S'}(\D,m). \]
To do so, we inductively construct the partition $(P_k)_k$. We start with a marked floor diagram $(\D_1,m_1)$ and we determine a set $P_1$ of marked floor diagrams such that $P_1$ contains $(\D_1,m_1)$ and
\[ \sum_{(\D,m) \in P_1} \mu_S(\D,m) = \sum_{(\D,m) \in P_1} \mu_{S'}(\D,m). \]
We then choose another marked floor diagram $(\D_2,m_2) \notin P_1$, and similarly determine a set $P_2$ disjoint from $P_1$, etc. Hence, given an arbitrary marked floor diagram it suffices to give the part $P$ of the partition it is contained in.
More precisely, in the proof we introduce \emph{partial markings} and we will simultaneously handle the case of several marked diagrams, all coming from the same \emph{partial} marked diagram. 

\begin{demo}[Proof of theorem \ref{theo-invS}]
It is sufficient to suppose that $S$ and $S'$ differ by one pair, and we can assume that this pair is $\{i,i+1\} \in S$ and $\{i+1,i+2\}\in S'$. 
Given $\D$ a floor diagram of Newton polygon $\Delta$ and genus $g$, a \emph{partial marking} of $\D$ is an application that associates to all but three elements of $\D$ an integer of $\{1,\dots,n(\D)\} \setminus \{i,i+1,i+2\}$ in a bijective and increasing way. A partial marking gives several markings by labeling the three remaining elements of $\D$ with $i$, $i+1$ and $i+2$. 

Let $\D$ be a floor diagram of Newton polygon $\Delta$ and genus $g$. Assume we are given a partial marking of $\D$. We will investigate the possibilities to construct a marked floor diagram from this data. To do so, for any relative positions of the three elements left aside by the partial marking, we look at the possible choices to extend the partial marking. We will distinguish cases according to the number of vertices left aside by the partial marking. In all the proof, $W$ will be the contribution to $\muSDm$ and $\muSSDm$ of the edges marked by the partial marking. 

\tocless{\subsection*{3 vertices.}} 
 In that case both $S$ and $S'$ are incompatible whatever the marking $m$ extending the partial marking is, \ie 
 \[ \muSDm = \muSSDm = 0.\]
 So take $P = \{(\D,m),\ m \text{ extension of the partial marking} \}$.

\tocless{\subsection*{2 vertices.}}
The unique edge left aside by the partial marking can :
\begin{itemize}
    \item link the two vertices (figure \ref{fig-2vert-a}),
    \item be adjacent to only one of the two edges (figure \ref{fig-2vert-b}, and the symmetric case where the edge is above the vertex),
    \item or be adjacent to none of the vertices (figure \ref{fig-2vert-c}).
\end{itemize} 
On those pictures we do not represent other vertices and edges of $\D$. 

\begin{figure}[h!]
	\begin{subfigure}[t]{0.33\textwidth}
	\centering
	\begin{tikzpicture}[scale=1]	
	
		\floor (1) at (0,0) {} ;
		\floor (2) at (0,1.5) {} ;
		\draw (1) to (2) ;
		
	\end{tikzpicture}
	\caption{}
	\label{fig-2vert-a}
\end{subfigure}
\begin{subfigure}[t]{0.3\textwidth}
	\centering
	\begin{tikzpicture}[scale=1]
	
	\floor (1) at (0,1) {} ;
	\draw (1) to (0,0) ;
	\floor at (1.5,0.5) {} ;

	\end{tikzpicture}
	\caption{}
	\label{fig-2vert-b}
\end{subfigure}	
\begin{subfigure}[t]{0.33\textwidth}
	\centering
	\begin{tikzpicture}[scale=1]
		
	\draw (0,1.5) to (0,0) ;
	\floor at (1.5,0.75) {} ;
	\floor at (3,0.75) {} ;		
		
	\end{tikzpicture}
	\caption{}
	\label{fig-2vert-c}
\end{subfigure}
\caption{Possible configurations with 2 vertices.}
	\label{2vert}
\end{figure}
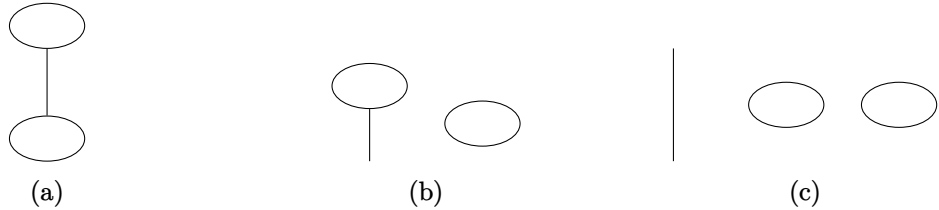

We deal with the three cases separately.

\begin{enumerate}[label=(\alph*)]

\item There is only one possible marking $m$ and one has 
\[ \muSDm = \muSSDm \] 
so take $P =\{ (\D,m) \}$.

\item There are three possible markings. Let $m_k$ be the extension where the right vertex is $i+k$ for $k=0,1,2$. The marking $m_1$ is incompatible with both $S$ and $S'$ \ie $\muSDmkk = \muSSDmkk = 0$, and one has $\muSDmk = \muSSDmkkk = 0$ and $\muSDmkkk = \muSSDmk$. Thus
\[  \muSDmk + \muSDmkk + \muSDmkkk = \muSSDmk + \muSSDmkk + \muSSDmkkk \]
and we take $P =\{ (\D,m_0),(\D,m_1),(\D,m_2) \}$.

\item Any marking $m$ is incompatible with both $S$ and $S'$ \ie  
\[ \muSDm = \muSSDm = 0 \]
and take $P = \{(\D,m),\ m \text{ extension of the partial marking} \}$.

\end{enumerate}

\tocless{\subsection*{1 vertex.}}
The unique vertex left aside by the partial marking can :
\begin{itemize}
    \item be adjacent to both edges (figure \ref{fig-1vert-a} where the edges can share a second common vertex or not, the symmetric case where the edges are above the vertex, and figure \ref{fig-1vert-b}),

    \item be adjacent to one of the two edges (figures \ref{fig-1vert-c} where the edges are adjacent to a common vertex, the symmetric case where the common vertex is above the edges, figure \ref{fig-1vert-d} and its symmetric case),

    \item or be adjacent to none of the edges (figure \ref{fig-1vert-e} if the edges are adjacent to at least one common vertex, its symmetric case, and figure \ref{fig-1vert-f}).
\end{itemize}

On those pictures, solid lines are for elements left aside by the partial marking, and we represent other vertices with dashed lines if they are relevant (\ie play a role) in the calculations.

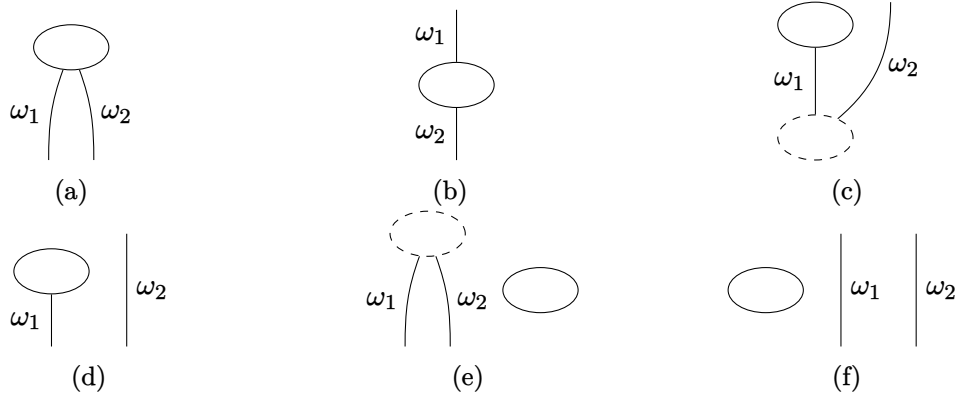
\begin{figure}[h!]
	\begin{subfigure}[t]{0.3\textwidth}
	\centering
	\begin{tikzpicture}[scale=1]
	
	\floor (1) at (0,1.5) {} ;
	\draw (1) to[out=-70, in = 90] node[right,pos=1/2]{$\omega_2$} (0.3,0) ;
	\draw (1) to[out=-110, in = 90] node[left,pos=1/2]{$\omega_1$} (-0.3,0) ;

	\end{tikzpicture}
	\caption{}
	\label{fig-1vert-a}
\end{subfigure}	
\begin{subfigure}[t]{0.33\textwidth}
	\centering
	\begin{tikzpicture}[scale=1]	
	
		\floor (1) at (0,0) {} ;
		\draw (1) to node[left,pos=1/2]{$\omega_1$} (0,1) ;
		\draw (1) to node[left,pos=1/2]{$\omega_2$} (0,-1) ;
		
	\end{tikzpicture}
	\caption{}
	\label{fig-1vert-b}
\end{subfigure}
\begin{subfigure}[t]{0.33\textwidth}
	\centering
	\begin{tikzpicture}[scale=1]
	
	\floor (2) at (0,1.5) {} ;
	\floor[dashed] (1) at (0,0) {} ;
	\draw (1) to node[left,pos=1/2]{$\omega_1$} (2) ;
	\draw (1) to[out=40, in=-90] node[right,pos=1/2]{$\omega_2$} (1,1.8) ;

		
	\end{tikzpicture}
	\caption{}
	\label{fig-1vert-c}
\end{subfigure}
\begin{subfigure}[t]{0.33\textwidth}
	\centering
	\begin{tikzpicture}[scale=1]
	
	\floor (1) at (0,1) {} ;
	\draw (1) to node[left,pos=1/2]{$\omega_1$} (0,0) ;
	\draw (1,0) to node[right,pos=1/2]{$\omega_2$} (1,1.5) ;	
		
	\end{tikzpicture}
	\caption{}
	\label{fig-1vert-d}
\end{subfigure}
\begin{subfigure}[t]{0.3\textwidth}
	\centering
	\begin{tikzpicture}[scale=1]
		
	\floor at (3,0.75) {} ;
	\floor[dashed] (1) at (1.5,1.5) {} ;
	\draw (1) to[out=-70, in=90] node[right,pos=1/2]{$\omega_2$} (1.8,0) ;
	\draw (1) to[out=-110, in=90] node[left,pos=1/2]{$\omega_1$} (1.2,0) ;

	\end{tikzpicture}
	\caption{}
	\label{fig-1vert-e}
\end{subfigure}	
\begin{subfigure}[t]{0.33\textwidth}
	\centering
	\begin{tikzpicture}[scale=1]	
		
	\draw (1,1.5) to node[right,pos=1/2]{$\omega_1$} (1,0) ;
	\draw (2,1.5) to node[right,pos=1/2]{$\omega_2$} (2,0) ;
	\floor at (0,0.75) {} ;
		
	\end{tikzpicture}
	\caption{}
	\label{fig-1vert-f}
\end{subfigure}

\caption{Possible configurations with 1 vertex.}
	\label{fig-1vert}
\end{figure}

We deal with the different cases separately.

\begin{enumerate}[label=(\alph*)]
\item Denote $m_0$ (resp. $m_1$) the marking where the left edge is $i$ (resp. $i+1$). Then one has $\muSDmk = \muSDmkk = \dfrac{[\omega_1][\omega_2][\omega_1+\omega_2]}{[2]} W$, $\muSSDmk = [\omega_1]^2 [\omega_2]_2 W$ and $\muSSDmkk = [\omega_1]_2 [\omega_2]^2 W$. Lemma \ref{lem-quantum} shows that 
\[ \muSDmk + \muSDmkk = \muSSDmk + \muSSDmkk  \]
so take $P = \{(\D,m_0), (\D,m_1)\}$.

\item[(b,c)] If the diagram $\D$ is in case (b), it might be necessary to include marked diagrams of case (c) to the part $P$ containing $(\D,m)$, where $m$ is the unique marking extending the partial marking of $\D$. For that reason, cases (b) and (c) are handled together.

In case (b), first assume $\omega_1=\omega_2$. Then one has $\muSDm = \muSSDm$ and we take $P=\{(\D,m)\}$ ; when $\omega_1 = \omega_2$ there is no corresponding diagram of case (c).

Otherwise $\omega_1 \neq \omega_2$ and we assume $\omega_2 > \omega_1$. In particular, one has $\omega_2>1$ so the edge with weight $\omega_2$ cannot be an infinite edge and is necessarily adjacent to a second vertex. Moreover $\omega_2$ can be written $\omega_1 + (\omega_2-\omega_1)$ with both terms positive. In the end, this case (b) is related to case (c) via an operation $A^+$, see figure \ref{fig-1vertA+}. 
Conversely any case (c) gives a case (b) with $\omega_2 > \omega_1$ via an operation $A^+$.

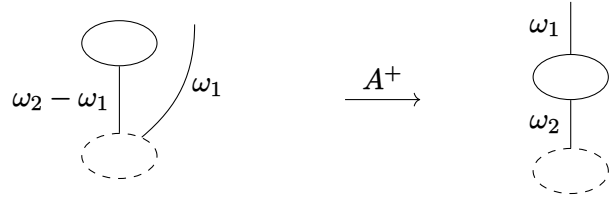
\begin{figure}[h]
	\centering
\begin{tikzpicture}[scale=1]	
	\floor (2) at (0,0.75) {} ;
	\floor[dashed] (1) at (0,-0.75) {} ;
	\draw (1) to node[left,pos=1/2]{$\omega_2-\omega_1$} (2) ;
	\draw (1) to[out=40, in=-90] node[right,pos=1/2]{$\omega_1$} (1,1) ;	
	
	\draw[->] (3,0) -- node[above,pos=1/2]{$A^+$} (4,0)  ;
	
	\floor (2) at (6,0.3) {} ;
    \floor[dashed] (1) at (6,-0.95) {} ;
	\draw (2) to node[left,pos=1/2]{$\omega_1$} (6,1.3) ;
	\draw (2) to node[left,pos=1/2]{$\omega_2$} (1) ;
		
\end{tikzpicture}
\caption{Passing from case (c) to case (b).}
	\label{fig-1vertA+}
\end{figure}

Let $\D'$ be the floor diagram of case (c) which gives the diagram $\D$ of case (b) with the $A^+$ operation of figure \ref{fig-1vertA+}. Let $m'_k$ be the marking of $\D'$ where the right edge is $i+k$ for $k=0,1,2$. 
One has :

\begin{itemize}
 \item $\muSDDmk = \dfrac{[\omega_1][\omega_2-\omega_1][\omega_2]}{[2]} W$ and $\muSSDDmk = [\omega_1]^2 [\omega_2-\omega_1]_2 W$,
 
 \item $\muSDDmkk = \dfrac{[\omega_1][\omega_2-\omega_1][\omega_2]}{[2]} W$ and $\muSSDDmkk = 0$,
 
 \item $\muSDDmkkk = [\omega_1]^2 [\omega_2-\omega_1]_2 W$ and $\muSSDDmkkk = 0$.
\end{itemize}

For $\D$ we have  $\muSDm = [\omega_1]^2 [\omega_2]_2 W$ and $\muSSDm = [\omega_1]_2 [\omega_2]^2 W$. Hence taking $a=\omega_1$ and $b=\omega_2 - \omega_1$ in lemma \ref{lem-quantum} we see that 
   \begin{align*}
       &\ \muSDm + \muSDDmk + \muSDDmkk + \muSDDmkkk \\
       =&\ \muSSDm + \muSSDDmk + \muSSDDmkk + \muSSDDmkkk
   \end{align*}
and we can take $P=\{ (\D,m), (\D',m'_0), (\D',m'_1), (\D',m'_2) \}$.
If $\omega_1 > \omega_2$ the proof is analogous using the symmetric case of figure \ref{fig-1vert-c} and an operation $A^-$.

\item[(d)] This is similar to figure \ref{fig-2vert-b}. There are three possible markings. Let $m_k$ be the extension where the right edge is $i+k$ for $k=0,1,2$. The marking $m_1$ is incompatible with both $S$ and $S'$ \ie $\muSDmkk = \muSSDmkk = 0$, and one has $\muSDmk = \muSSDmkkk = 0$ and $\muSDmkkk = \muSSDmk$ so
\[  \muSDmk + \muSDmkk + \muSDmkkk = \muSSDmk + \muSSDmkk + \muSSDmkkk \]
and take $P =\{ (\D,m_0),(\D,m_1),(\D,m_2) \}$.

\item[(e)] Let $m_k$ and $m'_k$ be the two markings where the vertex is $i+k$ for $k=0,1,2$. One has
\[ \muSDmkk = \muSSDmkk = \muSDmk = \muSDmmk = \muSSDmkkk = \muSSDmmkkk = 0 \]
and 
\[ \muSDmkkk = \muSDmmkkk = \muSSDmk = \muSSDmmk = \dfrac{[\omega_1][\omega_2][\omega_1+\omega_2]}{[2]} W, \]
 so
\begin{align*}
    &\ \muSDmk + \muSDmmk + \muSDmkk + \muSDmmkk + \muSDmkkk + \muSDmmkkk  \\
    =&\ \muSSDmk + \muSSDmmk + \muSSDmkk + \muSSDmmkk + \muSSDmkkk + \muSSDmmkkk.
\end{align*} 
and we take $P = \{ (\D,m_0), (\D,m'_0), (\D,m_1), (\D,m'_1), (\D,m_2), (\D,m'_2) \}$.

\item[(f)] Any marking $m$ is incompatible with both $S$ and $S'$ \ie 
\[ \muSDm = \muSSDm = 0 \]
and take $P = \{(\D,m),\ m \text{ extension of the partial marking} \}$.

\end{enumerate}

\tocless{\subsection*{0 vertex.}}
The edges left aside by the partial marking can :
\begin{itemize}
    \item be adjacent to a common vertex (figure \ref{fig-0vert-a} and the symmetric case where the vertex is below the edge),

    \item one can share at least a common vertex with any of the others, but the other two do not have a common vertex (figure \ref{fig-0vert-b}),

    \item two of them can share at least one common vertex, and the last edge has no common vertex with the other two (figure \ref{fig-0vert-c}),

    \item have no common vertex (figure \ref{fig-0vert-d}).
\end{itemize}

On those pictures, solid lines are for elements left aside by the partial marking, and we represent other vertices with dashed lines if they are relevant (\ie play a role) in the calculations.

\begin{figure}[h]
\begin{subfigure}[t]{0.24\textwidth}
	\centering
	\begin{tikzpicture}[scale=1]	
	
		\floor[dashed,minimum width=1.5cm] (1) at (0,1.5) {} ;
		\draw (1) to[out=-145,in=90] node[left,pos=1/2]{$\omega_1$} (-0.7,0) ;
		\draw (1) to node[left,pos=1/2]{$\omega_2$} (0,0) ;
		\draw (1) to[out=-35,in=90] node[right,pos=1/2]{$\omega_3$} (0.7,0) ;
		
	\end{tikzpicture}
	\caption{}
	\label{fig-0vert-a}
\end{subfigure}
\begin{subfigure}[t]{0.24\textwidth}
	\centering
	\begin{tikzpicture}[scale=1]
	
		\floor[dashed] (1) at (0,1.5) {} ;
		\floor[dashed] (2) at (1,0) {} ;
		\draw (1) to[out=-110,in=90] node[left,pos=1/2]{$\omega_1$} (-0.3,0) ;
		\draw (1) to node[right,pos=1/3]{$\omega_2$} (2) ;
		\draw (2) to[out=70,in=-90] node[right,pos=1/2]{$\omega_3$} (1.3,1.8) ;

	\end{tikzpicture}
	\caption{}
	\label{fig-0vert-b}
\end{subfigure}	
\begin{subfigure}[t]{0.24\textwidth}
	\centering
	\begin{tikzpicture}[scale=1]
	
	\floor[dashed] (1) at (0,1.5) {} ;
	\draw (1) to[out=-70, in=90] node[right,pos=1/2]{$\omega_2$} (0.3,0) ;
	\draw (1) to[out=-110, in=90] node[left,pos=1/2]{$\omega_1$} (-0.3,0) ;
	\draw (1,0) to node[right,pos=1/2]{$\omega_3$} (1,1.8) ;
		
	\end{tikzpicture}
	\caption{}
	\label{fig-0vert-c}
\end{subfigure}
\begin{subfigure}[t]{0.24\textwidth}
	\centering
	\begin{tikzpicture}[scale=1]
	
	\draw (0,0) to node[left,pos=1/2]{$\omega_1$} (0,1.5) ;
	\draw (1,0) to node[left,pos=1/2]{$\omega_2$} (1,1.5) ;
	\draw (2,0) to node[left,pos=1/2]{$\omega_3$} (2,1.5) ;
		
	\end{tikzpicture}
	\caption{}
	\label{fig-0vert-d}
\end{subfigure}

\caption{Possible configurations with 0 vertex.}
	\label{fig-0vert}
\end{figure}
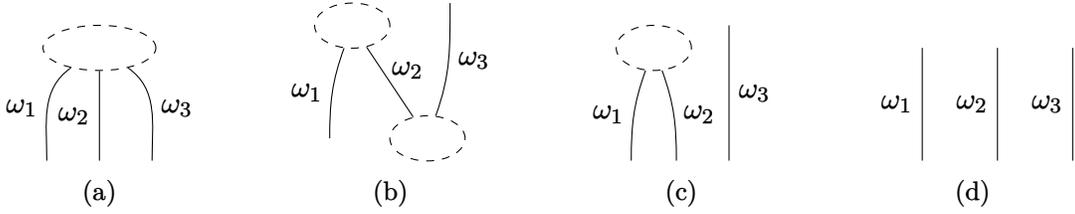

We deal with the different cases separately.

\begin{enumerate}[label=(\alph*)]
\item There are six possible markings. The contributions are summed up in table \ref{table-contrib-0vert-a}, where $(j,k,\ell)$ denotes the markings of the edges from left to right.

\begin{table}[h!]
\renewcommand{\arraystretch}{1.5}
\centering
\[ 
\begin{array}{|c||c|c|}
\hline
  & S  & S' \\
\hline
\hline
(i,i+1,i+2) & \dfrac{[\omega_1][\omega_2][\omega_1 + \omega_2]}{[2]} [\omega_3]^2 & \dfrac{[\omega_2][\omega_3][\omega_2 + \omega_3]}{[2]} [\omega_1]^2  \\
\hline 
(i,i+2,i+1) & \dfrac{[\omega_1][\omega_3][\omega_1 + \omega_3]}{[2]} [\omega_2]^2 & \dfrac{[\omega_2][\omega_3][\omega_2 + \omega_3]}{[2]} [\omega_1]^2   \\
\hline 
(i+1,i,i+2) & \dfrac{[\omega_1][\omega_2][\omega_1 + \omega_2]}{[2]} [\omega_3]^2 & \dfrac{[\omega_1][\omega_3][\omega_1 + \omega_3]}{[2]} [\omega_2]^2    \\
\hline 
(i+1,i+2,i) & \dfrac{[\omega_1][\omega_3][\omega_1 + \omega_3]}{[2]} [\omega_2]^2 & \dfrac{[\omega_1][\omega_2][\omega_1 + \omega_2]}{[2]} [\omega_3]^2   \\
\hline 
(i+2,i,i+1) & \dfrac{[\omega_2][\omega_3][\omega_2 + \omega_3]}{[2]} [\omega_1]^2 & \dfrac{[\omega_1][\omega_3][\omega_1 + \omega_3]}{[2]} [\omega_2]^2   \\
\hline 
(i+2,i+1,i) & \dfrac{[\omega_2][\omega_3][\omega_2 + \omega_3]}{[2]} [\omega_1]^2 & \dfrac{[\omega_1][\omega_2][\omega_1 + \omega_2]}{[2]} [\omega_3]^2    \\
\hline 
\end{array} \]
\caption{Contribution of the markings in case (a).}
\label{table-contrib-0vert-a}
\end{table}

The sums of the two columns are the same, so these marked floor diagrams give the same contributions to $G_g(\Delta,S)$ and $G_g(\Delta,S')$ and we take $P$ the set of these marked floor diagrams.

 Note that, depending on the unshown part of the diagram and on the precise value of the weights, some markings may give isomorphic marked diagrams : there may be only 3 or 1 marked floor diagram instead of 6. However, in that case some of the weight among $\omega_1$, $\omega_2$ and $\omega_3$ are equal, and removing the superfluous rows if the table does not affect the equality of the sums of the columns.

\item Similarly to the previous case we get table \ref{table-contrib-0vert-b}. We see that the sums of the two columns are the same, so these marked floor diagrams give the same contributions to $G_g(\Delta,S)$ and $G_g(\Delta,S')$ and we take $P$ the set of these marked floor diagrams.

\begin{table}[h!]
\renewcommand{\arraystretch}{1.5}
\centering
\[ 
\begin{array}{|c||c|c|}
\hline
  & S  & S' \\
\hline
\hline
(i,i+1,i+2) & \dfrac{[\omega_1][\omega_2][\omega_1 + \omega_2]}{[2]} [\omega_3]^2 & \dfrac{[\omega_2][\omega_3][\omega_2 + \omega_3]}{[2]} [\omega_1]^2  \\
\hline 
(i,i+2,i+1) & 0 & \dfrac{[\omega_2][\omega_3][\omega_2 + \omega_3]}{[2]} [\omega_1]^2   \\
\hline 
(i+1,i,i+2) & \dfrac{[\omega_1][\omega_2][\omega_1 + \omega_2]}{[2]} [\omega_3]^2 & 0   \\
\hline 
(i+1,i+2,i) & 0 & \dfrac{[\omega_1][\omega_2][\omega_1 + \omega_2]}{[2]} [\omega_3]^2   \\
\hline 
(i+2,i,i+1) & \dfrac{[\omega_2][\omega_3][\omega_2 + \omega_3]}{[2]} [\omega_1]^2 & 0  \\
\hline 
(i+2,i+1,i) & \dfrac{[\omega_2][\omega_3][\omega_2 + \omega_3]}{[2]} [\omega_1]^2 & \dfrac{[\omega_1][\omega_2][\omega_1 + \omega_2]}{[2]} [\omega_3]^2    \\
\hline 
\end{array} \]
\caption{Contribution of the markings in case (b).}
\label{table-contrib-0vert-b}
\end{table}

\item This is the same as in figure \ref{fig-1vert-e}. Let $m_k$ and $m'_k$ be the two markings where the right edge is $i+k$ for $k=0,1,2$. Both $m_1$ and $m'_1$ are incompatible with $S$ and $S'$, and the contributions of $m_0$ and $m'_0$ balance with those of $m_2$ and $m'_2$. Hence we take $P = \{ (\D,m_0), (\D,m'_0), (\D,m_1), (\D,m'_1), (\D,m_2), (\D,m'_2) \}$.

\item Any marking $m$ is incompatible with both $S$ and $S'$ \ie 
\[ \muSDm = \muSSDm = 0 \]
and take $P = \{(\D,m),\ m \text{ extension of the partial marking} \}$.
\end{enumerate}
\end{demo}

We can thus abusively write $G_g(\Delta,s)$ instead of $G_g(\Delta,S)$.

\begin{defi}\label{def-combi-GS-inv}
Let $\Delta$ be a $h$-transverse polygon, $g\in\N$, $s\in \N$ and $S$ be any pairing  of order $s$ of $\{1,\dots, 2\smax(\Delta,g) \}$. We define 
\[ G_g(\Delta,s) = \dsum_{(\D,m)} \mu_S(\D,m) \in \Z[q^{\pm1}] \]
where the sum runs over the isomorphism classes of marked floor diagrams with Newton polygon $\Delta$ and genus $g$.
The Laurent polynomial $G_g(\Delta,s)$ is called \emph{Göttsche-Schroeter (refined) invariant of genus $g$}.
\end{defi}

\subsection{Properties of the invariants}

In this section we prove few properties satisfied by the higher genus Göttsche-Schroeter invariant. We essentially adapt, when necessary, the proofs given by Brugallé and Jaramillo-Puentes in \cite{brugalle_polynomiality_2022} for the case of genus $0$ invariants.

\begin{prop}
Let $(\D,m)$ be a marked floor diagram of genus $g$, and $S_1 \subset S_2$ be two pairing of the set $\{1,\dots,n(\D)\}$. Then one has 
$\mu_{S_1}(\D,m) - \mu_{S_2}(\D,m) \in \N[q^{\pm 1}]$. 
\end{prop}

\begin{coro} \label{coro-decrease-with-s}
Let $\Delta$ be a $h$-transverse polygon and $g\in\N$. For any $i\in\N$ one has
\[ \ang{G_g(\Delta,0)}_i \geq \ang{G_g(\Delta,1)}_i \geq \dots \geq \ang{G_g(\Delta,\smax(\Delta,g))}_i . \]
\end{coro}

\begin{demo}[Proofs]
The proofs of \cite[proposition 2.16 and corollary 2.17]{brugalle_polynomiality_2022} rely on calculations on quantum integers, and the genus does not play any role. Thus we can copy their proofs. 
\end{demo}

The decrease with respect to $S$ for $\mu_S(\D,m)$, and with respect to $s$ for $G_g(\Delta,s)$ can be observed in the examples of section \ref{subsec-ex-combi-GS-inv}.

\begin{prop} \label{prop-welsch}
Let $\Delta$ be a $h$-transverse polygon whose top is depicted in figure \ref{fig-welscha} and $\tilde\Delta$ be the polygon obtained in figure \ref{fig-welschb} by cutting of the top corner of $\Delta$. 
If $s \leq \smax(\Delta,g)$, then 
\[ G_g(\Delta,s+1) = G_g(\Delta,s) - 2 G_g(\tilde \Delta,s) . \]
\end{prop}

\begin{figure}[h!]
	\begin{subfigure}[t]{0.49\textwidth}
	\centering
	\begin{tikzpicture}[scale=1/2]
		\draw (0,0) node {$\bullet$} to (0,3) node {$\bullet$} to (3,0) node {$\bullet$} ;
		\node at (0,1) {$\bullet$} ;
		\node at (0,2) {$\bullet$} ;
		\node at (2,1) {$\bullet$} ;
		\node at (1,2) {$\bullet$} ;
		\node at (1,1) {$\bullet$} ;
        \node at (1,0) {$\bullet$} ;
        \node at (2,0) {$\bullet$} ;
		\draw[dashed] (0,-1) to (0,0) ;
		\draw[dashed] (3,0) to (4,-1) ;
	\end{tikzpicture}
	\caption{$\Delta$}
	\label{fig-welscha}
\end{subfigure}	
\begin{subfigure}[t]{0.49\textwidth}
	\centering
	\begin{tikzpicture}[scale=1/2]
		\draw (0,0) node {$\bullet$} to (0,1) node {$\bullet$} to (2,1) node {$\bullet$} to (3,0) node {$\bullet$} ;
		\node at (0,2) {$\bullet$} ;
		\node at (1,2) {$\bullet$} ;
		\node at (1,1) {$\bullet$} ;
		\node at (0,3) {$\bullet$} ;
        \node at (1,0) {$\bullet$} ;
        \node at (2,0) {$\bullet$} ;
		\draw[dashed] (0,-1) to (0,0) ;
		\draw[dashed] (3,0) to (4,-1) ;
	\end{tikzpicture}
	\caption{$\tilde\Delta$}
	\label{fig-welschb}
\end{subfigure}
	\caption{}
	\label{fig-welsch}
\end{figure}

\begin{demo}
The proof is analogous to the one of \cite[proposition 2.19]{brugalle_polynomiality_2022}.
\end{demo}

We now extend \cite[theorem 1.7]{brugalle_polynomiality_2022} to arbitrary genus.

\begin{theo} \label{theo-GSinv-poly-s}
Let $\Delta$ be a $h$-transverse polygon and $g \leq \gmax(\Delta)$. If $2i \leq e^{-\infty}(\Delta)$ and $i \leq \gmax(\Delta)$, then the values $\< G_g(\Delta,s) \>_i$ for $0 \leq s \leq \smax(\Delta,g)$ are interpolated by a polynomial of degree $i$, whose leading coefficient is $\frac{(-2)^i}{i!} \binom{g_{\max}-i}{g}$.
\end{theo}

\begin{demo}
The beginning of the proof is as in \cite[theorem 1.7]{brugalle_polynomiality_2022}, hence we will not give the details of the computations before step 2(b) below. Let first introduce few notations.

We denote by $a_i$ the polynomial of degree at most $\smax = \smax(\Delta,g)$ which interpolates the values $(\< G_g(\Delta,s) \>_i)_{0 \leq s \leq \smax}$. Its $i$-th discrete derivative $a_i^{(i)}$ has degree at most $\smax-i$, and we want to show that 
\[ a_i^{(i)}(0) = \dots = a_i^{(i)}(s_{\max}-i) = 2^i \binom{\gmax-i}{g} .\]

	Let $0 \leq s \leq \smax-i$ and $S$ be a pairing of order $s$ of $\{ 2i+1,\dots,y(\Delta)-1+g \}$. For $I \subset \{1,\dots, i\}$ we denote 
\[ S^I = S \cup \bigcup_{j\in I} \{ \{ 2j-1,2j \}  \} \]
the pairing of order $s+|I|$ of $\{1,\dots,y(\Delta)-1+g\}$.
	Given $(\D,m)$ a marked floor diagram with Newton polygon $\Delta$ and genus $g$ we define
	\[ \kappa(\D,m) = \dsum_{\ell =1}^i \dsum_{\substack{I \subset \{1,\dots,i\} \\ |I| = \ell}} (-1)^\ell \mu_{S^I}(\D,m).  \]
	One has
	\[ \dsum_{j=-\gmax+g}^{\gmax-g} a^{(i)}_{\gmax-g-|j|}(s)q^j = \dsum_{(\D,m)} \kappa(\D,m)  \]
	where the sum runs over the isomorphism classes of marked floor diagrams of Newton polygon $\Delta$ and genus $g$. Hence the diagrams with degree at least $\gmax-g-i$, \ie codegree at most $i$, contribute to $a_i^{(i)}$.
	
	Let $(\D,m)$ be such a diagram. Denote by $i_0$ the minimal element of $\{1,\dots,n(\D)\}$ such that $m^{-1}(i_0) \in V(\D)$, and by $J \subset \{1,\dots,2i\}$ the set of elements $j$ such that $m^{-1}(j)$ is an elevator in $E^{-\infty}(\D)$ adjacent to $m^{-1}(i_0)$.

\tocless{\subsection*{Step 1.}} 
If $J\cup \{i_0\}$ contains a pair $\{2k-1,2k\}$ with $k\leq i$, then $\kappa(\D,m) = 0$.

We assume from now on that $J\cup \{i_0\}$ does not contain any pair $\{2k-1,2k\}$ with $k\leq i$. In particular, $|J| \leq i$.

\tocless{\subsection*{Step 2(a).}}
If $i_0 \leq 2i$ then $\kappa(\D,m)$ does not contribute to $a_i^{(i)}(s)$.

\tocless{\subsection*{Step 2(b).}}
Suppose now that $i_0 > 2i$. In particular, $m(\{1,\dots,2i\}) \subset E^{-\infty}(\D)$. Let $K \subset \{ 2i+1,\dots, y(\Delta)-1 +g \}$ be the set of elements $k$ such that $m(k)$ is an elevator in $E^{-\infty}(\D)$ adjacent to $m(i_0)$ ; one has $|K| \leq e^{-\infty}(\Delta) - 2i$. Hence by lemma \ref{lem-opeAB} one has
\[ \codeg(\D) \geq e^{-\infty}(\Delta)  - |J| - |K| \geq e^{-\infty}(\Delta)  - i - (e^{-\infty}(\Delta) -2i) = i \]
so $\D$ can contribute to $a_i^{(i)}(s)$ if and only if $\codeg(\D) = i-g$, which implies $|J| = i$ and $|K| = e^{-\infty}(\Delta) -2i$. Thus, $i$ elevators in $E^{-\infty}(\D)$ are not adjacent to $m(i_0)$ and they are the only elements creating codegree in $\D$. Hence, $\D$ contributes to $a_i^{(i)}(s)$ if and only if the following set of conditions is satisfied :

\begin{itemize}
\item the order $\prec$ is total on $V(\D)$,

\item elevators in $E^{+\infty}(\D)$ are all adjacent to the top floor,

\item $|J| = i$ and $J$ contains no pair $\{2k-1, 2k\}$,

\item $m(\{1,\dots,2i\}\setminus J)$ consists exactly of elevators in $E^{-\infty}(\D)$ adjacent to the second lowest floor,

\item $E^{-\infty}(\D) \setminus m(\{ 1,\dots,2i \})$ consists of elevators adjacent to the lowest floor,

\item the function $L : V(\D) \to \bleft(\Delta)$ and $R : V(\D) \to \bright(\Delta)$ are increasing,

\item any bounded edge is between two consecutive vertices, \ie the genus is created only by configurations of figure \ref{genus} ; there is no configuration of figure \ref{genusno}. 
\end{itemize}

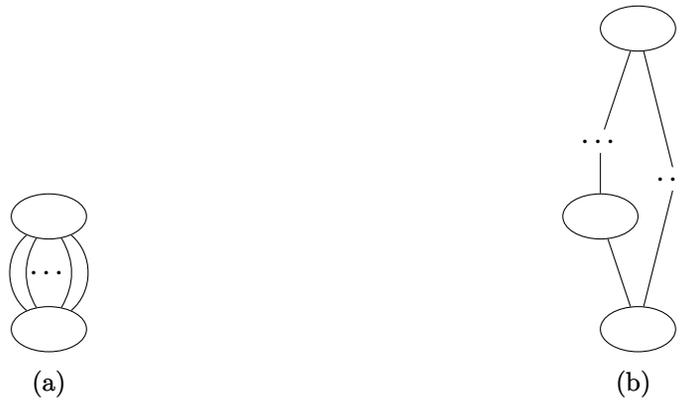
\begin{figure}[h!]
	\begin{subfigure}[t]{0.49\textwidth}
	\centering
	\begin{tikzpicture}[scale=1]	
	
		\floor (1) at (0,0) {} ;
		\floor (2) at (0,1.5) {} ;
		\draw (1) to[out=40, in=-40] (2) ;
		\draw (1) to[out=60, in=-60] (2) ;
		\draw (1) to[out=120, in=-120] (2) ;
		\draw (1) to[out=140, in=-140] (2) ;
		\node at (0,0.75) {$\dots$} ;
				
	\end{tikzpicture}
	\caption{}
	\label{genus}
\end{subfigure}
\begin{subfigure}[t]{0.49\textwidth}
	\centering
	\begin{tikzpicture}[scale=1]
	
		
	
		\floor (1) at (0,0) {} ;
		\floor (2) at (-0.5,1.5) {} ;
		\node (3) at (-0.5,2.5) {$\dots$} ;
        \node (4) at (0.5,2) {$\dots$} ;
		\floor (5) at (0,4) {} ;
  
		\draw (1) to (2) ;
        \draw (2) to (3) ;
        \draw (3) to (5) ;
        \draw (1) to (4) ;
        \draw (4) to (5) ;

	\end{tikzpicture}
	\caption{}
	\label{genusno}
\end{subfigure}	
	\caption{Possible configuration for the genus.}
	\label{gen}
\end{figure}

	The first conditions are those of \cite{brugalle_polynomiality_2022}, and the last is added to take into account the genus.
	These conditions ensure that the marked floor diagrams which contribute to $a_i^{(i)}(s)$ all satisfy $\kappa(\D,m) = \muSDm$ and have the shape depicted in figure \ref{fig-formediag}, where $a=a(\Delta)$ is the number of vertices.

\begin{figure}[h!]
	\centering
	\begin{tikzpicture}[scale=1]
		\floor (1) at (0,1.5) {\scriptsize $v_1$} ;
		\floor (2) at (1.5,3) {\scriptsize $v_2$} ;
		\floor (3) at (1.5,4.5) {\scriptsize $v_3$} ;
		\floor (4) at (1.5,6.5) {} ;
		\floor (5) at (1.5,8) {\scriptsize $v_a$} ;
		
		\draw (1) to[out=30, in=-120] (2) ;
		\draw (1) to[out=80, in=-150] (2) ;
		\node at (0.7,9/4) {\scriptsize $\dots$} ;
		
		\draw (2) to[out=60, in=-60] (3) ;
		\draw (2) to[out=120, in=-120] (3) ;
		\node at (1.5,15/4) {\scriptsize $\dots$} ;
		
		\draw[dashed] (3) to (1.5,5.1) ;
		\draw[dashed] (4) to (1.5,5.9) ;
		\node at (1.5,5.5) {\scriptsize $\dots$} ;
		
		\draw (4) to[out=60, in=-60] (5) ;
		\draw (4) to[out=120, in=-120] (5) ;
		\node at (1.5,29/4) {\scriptsize $\dots$} ;
		
		\draw (1) to[out=-120, in=90] (-0.4,0) ;
		\draw (1) to[out=-60, in=90]  (0.4,0);
		\node at (0,0.4) {\scriptsize $\cdots$} ;
		\node at (0,-0.4) {\large $\underbrace{\ \ \ \ \ \ }_{d_b\Delta-i}$} ;
		
		\draw (2) to[out=-110, in=90] (1.1,0) ;
		\draw (2) to[out=-70, in=90]  (1.9,0);
		\node at (1.5,0.4) {\scriptsize $\cdots$} ;
		\node at (1.5,-0.4) {\large $\underbrace{\ \ \ \ \ \ }_{i}$} ;
		
		\draw (5) to[out=120, in=-90] (1.1,9.5) ;
		\draw (5) to[out=60, in=-90]  (1.9,9.5);
		\node at (1.5,9.1) {\scriptsize $\cdots$} ;
		
	\end{tikzpicture}
	\caption{The diagrams that contributes to $a_i^{(i)}(s)$.}
	\label{fig-formediag}
\end{figure}
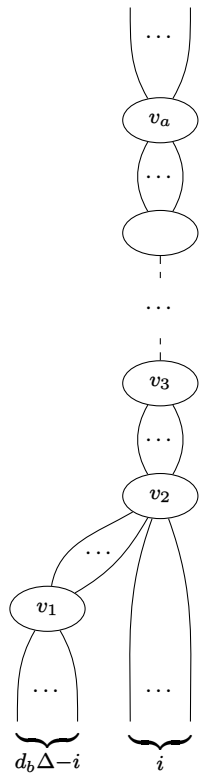

	There are $2^i$ possible choices for $J$, and given a $J$ it remains to determine how many marked diagrams of genus $g$ have a marking that corresponds to $J$. Starting with the unique marked diagram $(\D_0,m_0)$ of genus 0 corresponding to $J$, we need to 	choose a decomposition $g=g_1 + \dots + g_{a-1}$, and then split the unique edge between $v_j$ and $v_{j+1}$ in $g_j+1$ edges. If the weight of the edge is $w_j$, then there are $\binom{w_j-1}{g_j}$ ways to divide the weight and to mark  the new edges. Hence the total number of marked diagrams for a given $J$ is
\[ \sum_{\substack{g_1+\cdots + g_{a-1}=g \\ g_j \geq 0,\text{ ordered} }} \prod_{k=1}^{a-1} \binom{w_k-1}{g_k}  \]
which is just
\[ \binom{\dsum_{k=1}^{a-1}(w_k-1)}{g} = \binom{\deg(\D_0)}{g} = \binom{\gmax - \codeg(\D_0)}{g} = \binom{\gmax-i}{g}. \]
Hence the total number of marked diagrams is $2^i \binom{\gmax-i}{g}$.
Since the dominant coefficients of the multiplicities are 1 we conclude.
\end{demo}

\subsection{Link with other invariants}

In this section we show that the combinatorial Göttsche-Schroeter invariant matches the invariant of \cite[theorem 6.7]{shustin-2024-refined}. 
We refer to \cite[section 6]{shustin-2024-refined} for the definition of $RB_q(\Delta,g,(n_1,n_2))$, especially definition 6.2 for the multiplicity and remark 6.3. Note that this invariant is a count of tropical curves, while the combinatorial Göttsche-Schroeter invariant is a count of floor diagrams.

Remember we denote 
\[ [n](q) = \dfrac{q^{n/2}-q^{-n/2}}{q^{1/2}-q^{-1/2}},\ [n]^2 = [n](q)^2 \text{ and } [n]_2 = [n](q^2). \]
We also set
\[ \{n\}(q) = \dfrac{q^{n/2}+q^{-n/2}}{q^{1/2}+q^{-1/2}} . \]

\begin{prop}\label{prop-equal-shustin}
Let $\Delta$ be a $h$-transverse polygon, $g\in\N$ and $s\in\N$. The combinatorial Göttsche-Schroeter invariant corresponds to the invariant of \cite{shustin-2024-refined}, \ie
\[ G_g(\Delta,s)(q) = RB_q(\Delta,g,(y(\Delta)-1+g-2s,s)) . \]
\end{prop}

\begin{proof}
We will show a correspondence between the multiplicities used to compute both quantities. To do so, we examine the different terms appearing in the products that define both multiplicities.

\begin{figure}[h!]
	\begin{subfigure}[t]{0.49\textwidth}
	\centering
	\begin{tikzpicture}[scale=1]	
	
		\floor (1) at (0,0) {} ;
		\floor (2) at (0,1.5) {} ;
		\draw (1) to node[left,pos=1/2]{$\omega$} (2) ;
		
	\end{tikzpicture}
	\caption{}
	\label{fig-shustin1-a}
\end{subfigure}
\begin{subfigure}[t]{0.49\textwidth}
	\centering
	\begin{tikzpicture}[scale=1]
	
	\draw (0,1) to (1,1) to (2,2) ;
	\draw (1,1) to node[left,pos=1/2]{$\omega$} (1,0) ;
	\draw (0,0) to (1,0) to (2,-1) ;

	\end{tikzpicture}
	\caption{}
	\label{fig-shustin1-b}
\end{subfigure}	
	\caption{}
	\label{fig-shustin1}
\end{figure}

The situation of figure \ref{fig-shustin1-a} where an edge is unpaired in a floor diagram corresponds to the situation of figure \ref{fig-shustin1-b} at the level of tropical curves. In the floor diagram, the edge contributes $[w]^2$ to the multiplicity, while in the tropical curve the two adjacent vertices contribute $[w]\times[w] = [w]^2$. Hence the contributions are the same.

\begin{figure}[h!]
	\begin{subfigure}[t]{0.49\textwidth}
	\centering
	\begin{tikzpicture}[scale=1]	
	
		\floor (1) at (0,0) {} ;
		\floor (2) at (0,1.5) {$\bullet$} ;
		\draw (1) to node[left,pos=1/2]{$\omega$} node {$\bullet$} (2) ;
		
	\end{tikzpicture}
	\caption{}
	\label{fig-shustin2-a}
\end{subfigure}
\begin{subfigure}[t]{0.49\textwidth}
	\centering
	\begin{tikzpicture}[scale=1]
	
	\draw (0,1) to (1,1) to (2,2) ;
	\node at (1,1) {$\bullet$} ;
	\draw (1,1) to node[left,pos=1/2]{$\omega$} (1,0) ;
	\draw (0,0) to (1,0) to (2,-1) ;

	\end{tikzpicture}
	\caption{}
	\label{fig-shustin2-b}
\end{subfigure}	
	\caption{}
	\label{fig-shustin2}
\end{figure}

The situation of figure \ref{fig-shustin2-a} where an edge is paired with an adjacent vertex in a floor diagram corresponds to the situation of figure \ref{fig-shustin2-b} in the tropical curve, where a vertex adjacent to the corresponding edge is marked. In the floor diagram, the edge contributes $[w]_2$ to the multiplicity, while in the tropical curve the two adjacent vertices contribute $[w]\times\{w\} = [w]_2$. Hence the contribution are the same.

\begin{figure}[h!]
	\begin{subfigure}[t]{0.49\textwidth}
	\centering
	\begin{tikzpicture}[scale=1]	
	
		\floor (1) at (0,0) {} ;
		\floor (2) at (0,1.5) {} ;
		\draw (1) to[out=120, in=-120] node[left,pos=1/2]{$\omega_1$} node {$\bullet$} (2) ;
		\draw (1) to[out=60, in=-60] node[right,pos=1/2]{$\omega_2$} node {$\bullet$} (2) ;
		
	\end{tikzpicture}
	\caption{}
	\label{fig-shustin3-a}
\end{subfigure}
\begin{subfigure}[t]{0.49\textwidth}
	\centering
	\begin{tikzpicture}[scale=1]
	
	\draw (0,1) to (1,1) to (2,2) ;
	\draw (1,1) to node[right,pos=1/2]{$\omega_1+\omega_2$} (1,0) ;
	\draw (0.95,0) to node[left,pos=1/2]{$\omega_1$} (0.95,-1) ;
	\draw (1.05,0) to node[right,pos=1/2]{$\omega_2$} (1.05,-1) ;
	\node at (1,-1) {$\bullet$} ;
	\draw (1,-1) to node[right,pos=1/2]{$\omega_1+\omega_2$} (1,-2) ;
	\draw (0,-2) to (1,-2) to (2,-3) ;

	\end{tikzpicture}
	\caption{}
	\label{fig-shustin3-b}
\end{subfigure}	
	\caption{}
	\label{fig-shustin3}
\end{figure}

The situation of figure \ref{fig-shustin3-a} where two edges with two common adjacent vertices are paired corresponds to the situation of figure \ref{fig-shustin3-b} where there is a centrally embedded cycle (see \cite{shustin-2024-refined} for the terminology). Assume first that $\omega_1 \neq \omega_2$. At the level of floor diagrams, there are two possible markings and so the contribution to the multiplicity is
\[ 2 \dfrac{[\omega_1][\omega_2][\omega_1+\omega_2]}{[2]}. \]
At the level of tropical curves, the contribution to the multiplicity is
\[ [\omega_1+\omega_2] \times  \dfrac{2}{[2]}\dfrac{[\omega_1][\omega_2]}{[\omega_1+\omega_2]} \times [\omega_1+\omega_2] = 2 \dfrac{[\omega_1][\omega_2][\omega_1+\omega_2]}{[2]} ,\]
hence the contributions are the same.

If $\omega_1 = \omega_2$ then there is a single marking of the floor diagram, so the factor $2$ does not appear in the multiplicity. This is balanced by the fact that there is now a non-trivial automorphism of the tropical curve, hence we should also divide by $2$ the contribution to the multiplicity of the tropical curve.

\begin{figure}[h!]
	\begin{subfigure}[t]{0.49\textwidth}
	\centering
	\begin{tikzpicture}[scale=1]	
	
		\floor (2) at (0,1.5) {} ;
		\draw (-0.4,0) to[out=90, in=-120] node[left,pos=1/2]{$\omega_1$} node {$\bullet$} (2) ;
		\draw (0.4,0) to[out=90, in=-60] node[right,pos=1/2]{$\omega_2$} node {$\bullet$} (2) ;
		
	\end{tikzpicture}
	\caption{}
	\label{fig-shustin4-a}
\end{subfigure}
\begin{subfigure}[t]{0.49\textwidth}
	\centering
	\begin{tikzpicture}[scale=1]
	
	\draw (0,1) to (1,1) to (2,2) ;
	\draw (1,1) to node[right,pos=1/2]{$\omega_1+\omega_2$} (1,0) ;
	\node at (1,0) {$\bullet$} ;
	\draw (0.95,0) to node[left,pos=1/2]{$\omega_1$} (0.95,-1) ;
	\draw (0,-1) to (0.95,-1) to (1.95,-2) ;	
	\draw (1.05,0) to node[right,pos=1/2]{$\omega_2$} (1.05,-2) ;
	\draw (0,-2) to (1.05,-2) to (2,-2.95) ;

	\end{tikzpicture}
	\caption{}
	\label{fig-shustin4-b}
\end{subfigure}	
	\caption{}
	\label{fig-shustin4}
\end{figure}

The situation of figure \ref{fig-shustin4-a} where two edges with a unique common adjacent vertex are paired corresponds to the situation of figure \ref{fig-shustin4-b}. At the level of floor diagrams, there are two possible markings and so the contribution to the multiplicity is
\[ 2 \dfrac{[\omega_1][\omega_2][\omega_1+\omega_2]}{[2]}. \]
At the level of tropical curves, the contribution to the multiplicity is
\[ [\omega_1+\omega_2] \times  \dfrac{2}{[2]} \times [\omega_1] \times [\omega_2] = 2 \dfrac{[\omega_1][\omega_2][\omega_1+\omega_2]}{[2]} \]
so the contributions are the same.

These are the only possibilities appearing in a floor diagram. By \cite[lemma 5.6]{shustin-2024-refined} these are also the only terms that appear when computing $RB_q(\Delta,g,(y(\Delta)-1+g-2s,s))$. Hence the multiplicities match, and the counts are equal.
\end{proof}

\begin{rk}
Let $n = y(\Delta)-1+g$. By proposition \ref{prop-equal-shustin} and \cite[corollary 6.9]{shustin-2024-refined}, the integer $G_g(\Delta,s)(1)$ corresponds to the number of curves with Newton polygon $\Delta$, passing through $n-2s$ points on the toric surface $X_\Delta$ and with a fixed tangent direction at $s$ prescribed points. 
Here is an heuristic explanation of this fact, which has been communicated to us by Erwan Brugallé.

Take $s=1$. On the toric surface $X=X_\Delta$ we choose $n-2$ points and another point with a prescribed direction, as in figure \ref{fig-blowupXtan}. Blowing-up this point gives a $(-1)$-curve $E_1'$ with a point on it that corresponds to the direction we chose. If we blow-up again we obtain a $(-2)$-curve $\bar E_1$ and a $(-1)$-curve $\bar E_2$. The number of curves of genus $g$ and class $\Delta$ on $X$, through the $n-2$ points, and passing through the last point with the prescribed direction is then equal to the number of curves  of genus $g$ on $\bar X$, through $n-2$ points, and intersecting $\bar E_2$ but not $\bar E_1$, \ie of class $\Delta-\bar E_1 - 2\bar E_2$. We denote this number by $N_g(\bar X,\Delta- \bar E_1 - 2\bar E_2)$.

\begin{figure}[h!]
	\centering
\begin{tikzpicture}

	\draw plot [smooth cycle] coordinates {(0,0) (2,-0.3) (4,0.5) (4,2) (1,2.5)} ;
	\node at (2,-1) {$X$} ;
	\node at (1.5,1.2) {$\bullet$} ;
	\draw[->] (1.5,1.2) to (2,1.7) ;
	
	
    \tikzset{shift={(5.5,0)}}

	\draw plot [smooth cycle] coordinates {(0,0) (2,-0.3) (4,0.5) (4,2) (1,2.5)} ;
	\node at (2,-1) {$X'$} ;
	\draw plot [smooth, tension=0.8] coordinates {(1,2) (2.5,1.2) (3,0.5)} ;
	\node at (2.5,1.2) {$\bullet$} ;
	\node at (1.5,2) {\scriptsize $-1$} ;
	\node at (3.3,0.7) {\scriptsize $E_1'$} ;
	
	
    \tikzset{shift={(5.5,0)}}

	\draw plot [smooth cycle] coordinates {(0,0) (2,-0.3) (4,0.5) (4,2) (1,2.5)} ;
	\node at (2,-1) {$\bar X$} ;
	
	\draw plot [smooth, tension=0.8] coordinates {(1,2) (2.5,1.2) (3,0.5)} ;
	\node at (1.5,2) {\scriptsize $-2$} ;
	\node at (3.3,0.7) {\scriptsize $\bar E_1$} ;
	
	\draw plot [smooth, tension=1] coordinates {(1.3,0.5) (2.5,1.2) (3.2,2)} ;
	\node at (1.3,0.7) {\scriptsize $-1$} ;
	\node at (3.4,2) {\scriptsize $\bar E_2$} ;
\end{tikzpicture}
	\caption{}
	\label{fig-blowupXtan}
\end{figure}

On $X$ one can also choose $n$ points ; we depict on figure \ref{fig-blowupXpoint} the two added points compared with the previous situation. Blowing-up these points, we obtain two $(-1)$-curves $E_1$ and $E_2$. The number of curves of genus $g$ and class $\Delta$ on $X$ through the $n$ points is then equal to the number of curves on $\tilde X$ through $n-2$ points and intersecting $E_1$ and $E_2$, \ie of class $\Delta-E_1-E_2$. We denote this number by $N_g(\tilde X,\Delta-E_1-E_2)$.

\begin{figure}[h!]
	\centering
\begin{tikzpicture}
\centering

	\draw plot [smooth cycle] coordinates {(0,0) (2,-0.3) (4,0.5) (4,2) (1,2.5)} ;
	\node at (2,-1) {$X$} ;
	\node at (1.5,1) {$\bullet$} ;
	\node at (2.5,1.2) {$\bullet$} ;
	
	
    \tikzset{shift={(8,0)}}

	\draw plot [smooth cycle] coordinates {(0,0) (2,-0.3) (4,0.5) (4,2) (1,2.5)} ;
	\node at (2,-1) {$\tilde X$} ;
	
	\draw plot [smooth, tension=0.8] coordinates {(0.7,1.5) (1.5,1) (2.2,0.3)} ;
	\node at (0.7,1.2) {\scriptsize $-1$} ;
	\node at (1.9,0.3) {\scriptsize $E_1$} ;
	
	\draw plot [smooth, tension=0.8] coordinates {(1,2) (2.5,1.2) (3,0.5)} ;
	\node at (1.5,2) {\scriptsize $-1$} ;
	\node at (3.3,0.7) {\scriptsize $E_2$} ;
\end{tikzpicture}
	\caption{}
	\label{fig-blowupXpoint}
\end{figure}

Under degeneration, $E_2$ corresponds to $\bar E_2$ and $E_1$ corresponds to $\bar E_1 + \bar E_2$. The Abramovich-Bertram formula \cite{abramovich-2001-formula, vakil-2000-counting, brugalle-2020-invariance, bousseau_refined_2021} states that
\[ N_g(\bar X, \Delta-\bar E_1 -2\bar E_2) = N_g(\tilde X, \Delta-E_1-E_2) -2 N_g(\tilde X,\Delta-2E_1) .\]

One can reason similarly for any $s$, and this shows that the numbers of curves with $s+1$ tangency conditions can be calculated from the numbers of curves with $s$ tangency conditions, and recursively from the numbers of curves without tangency condition. 
But these last numbers correspond to $G_g(\Delta,0)(1)$, and we know that the invariants $G_g(\Delta,s)$ satisfy the formula of proposition \ref{prop-welsch}.
In particular, their values at $q=1$ also satisfy this recursive formula and have the same initial (with respect to $s\in\N$) values as the number of curves with point and tangency conditions. Hence the evaluations at $q=1$ of the combinatorial Göttsche-Schroeter invariants recover some numbers of curves on toric surfaces.
\end{rk}

\section{Examples and conjectures}

\subsection{Some calculations} \label{subsec-ex-combi-GS-inv}

In this section we run the calculations on some examples. When possible, we use theorem \ref{theo-GSinv-poly-s} to compute $G_g(\Delta,s)$ for few values of $s$ before interpolating. Otherwise, we compute $G_g(\Delta,s)$ for $0\leq s \leq \smax(\Delta,g)$. However, in our examples we notice that $\< G_g(\Delta,s) \>_i$ is always given by a polynomial of degree $i$ in $s$, even when theorem \ref{theo-GSinv-poly-s} does not apply. 
We use some tables to present the computations. In a column corresponding to a floor diagram we indicate its contribution to $G_g(\Delta,s)$. We put a $\star$ when this contribution does not change passing from $s$ to $s+1$, to highlight which diagrams contribute to the decrease of $G_g(\Delta,s)$ with respect so $s$, see corollary \ref{coro-decrease-with-s}. Note that for $g=\gmax(\Delta)$ one always has $G_{\gmax(\Delta)}(\Delta,s)=1$. Also, because the refined invariants are symmetric we do not precise the coefficients of the negative exponents. In all this section we use the pairing $S = \{\{1,2\},\dots,\{2s-1,2s\}\}$. We essentially deal with some examples where the Newton polygon is $\Delta^n_{a,b}$ for some special values of $(n,a,b)$, see figure \ref{fig-trapeze-Hirz}.

\begin{figure}[h]
	\centering
\begin{tikzpicture}[scale=1] 
		\draw (0,0) node {$\bullet$} ;
		\draw (0,0) node[below] {\scriptsize $(0,0)$} ;
		\draw (5,0) node {$\bullet$} ;
		\draw (5,0) node[below] {\scriptsize $(an+b,0)$} ;
		\draw (1,2) node {$\bullet$} ; 
		\draw (1,2) node[right] {\scriptsize $(b,a)$} ;
		\draw (0,2) node {$\bullet$} ; 
		\draw (0,2) node[left] {\scriptsize $(0,a)$} ; 
		\draw (0,0) -- (5,0) -- (1,2) -- (0,2) -- cycle ;
	\end{tikzpicture}
    \caption{The trapezoid $\Delta^n_ {a,b}$.}
    \label{fig-trapeze-Hirz}
\end{figure}
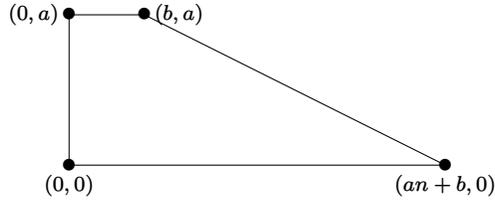

\begin{ex} \label{ex-square32}
We compute $G_g(\Delta^0_{3,2},s)$ for $0 \leq g \leq 2$. Tables \ref{table-square32g1} and \ref{table-square32g0} give
\begin{align*}
G_0(\Delta^0_{3,2},s) &= q^2 + (12-2s)q + (2s^2-22s+70) + \dots  \\
G_1(\Delta^0_{3,2},s) &= 2q + (16-2s) + \dots \\
G_2(\Delta^0_{3,2},s) &= 1 .
\end{align*}
\end{ex}

\begin{table}[h!]
\renewcommand{\arraystretch}{1.5}
\centering
\begin{tabular}{|c||c|c|c|c|c||c|}
\hline 
$s$ &
\centering
\begin{tikzpicture}[scale=3/4]
	\floor (1) at (0,1.5) {} ;
	\floor (2) at (0,3) {} ;
	\floor (3) at (0,4.5) {} ;
	
	\draw (1) to node[right,pos=1/2]{\scriptsize $2$} (2) ;
	
	\draw (2) to[out=120, in=-120] (3) ;
	\draw (2) to[out=60, in=-60] (3) ;
	
	\draw (1) to[out=-120, in = 90] (-0.4,0) ;
	\draw (1) to[out=-60, in = 90] (0.4,0) ;
	
	\draw (3) to[out=120, in = -90] (-0.4,6) ;
	\draw (3) to[out=60, in = -90] (0.4,6) ;
\end{tikzpicture}
	 & 
\centering
\begin{tikzpicture}[scale=3/4]	
	\floor (1) at (0,1.5) {} ;
	\floor (2) at (0,3) {} ;
	\floor (3) at (0,4.5) {} ;
	
	\draw (2) to node[right,pos=1/2]{\scriptsize $2$} (3) ;
	
	\draw (1) to[out=120, in=-120] (2) ;
	\draw (1) to[out=60, in=-60] (2) ;
	
	\draw (1) to[out=-120, in = 90] (-0.4,0) ;
	\draw (1) to[out=-60, in = 90] (0.4,0) ;
	
	\draw (3) to[out=120, in = -90] (-0.4,6) ;
	\draw (3) to[out=60, in = -90] (0.4,6) ;
\end{tikzpicture}
	 & 
\centering
\begin{tikzpicture}[scale=3/4]
	\floor (1) at (0,1.5) {} ;
	\floor (2) at (0.5,3) {} ;
	\floor (3) at (0,4.5) {} ;
	
	\draw (1) to (2) ;
	\draw (1) to[out=120, in=-120] (3) ;
	\draw (2) to (3) ;
	
	\draw (1) to[out=-120, in = 90] (-0.4,0) ;
	\draw (1) to[out=-60, in = 90] (0.4,0) ;
	
	\draw (3) to[out=120, in = -90] (-0.4,6) ;
	\draw (3) to[out=60, in = -90] (0.4,6) ;			
\end{tikzpicture}
	 & 
\centering
\begin{tikzpicture}[scale=3/4]	
	\floor (1) at (-0.5,1.5) {} ;
	\floor (2) at (0,3) {} ;
	\floor (3) at (0,4.5) {} ;
	
	\draw (1) to (2) ;
	
	\draw (2) to[out=120, in=-120] (3) ;
	\draw (2) to[out=60, in=-60] (3) ;
	
	\draw (1) to (-0.5,0) ;
	\draw (2) to[out=-70, in = 90] (0.5,0) ;
	
	\draw (3) to[out=120, in = -90] (-0.4,6) ;
	\draw (3) to[out=60, in = -90] (0.4,6) ;		
\end{tikzpicture}
	 & 
\centering
\begin{tikzpicture}[scale=3/4]	
	\floor (1) at (0,1.5) {} ;
	\floor (2) at (0,3) {} ;
	\floor (3) at (0.5,4.5) {} ;
	
	\draw (2) to (3) ;
	
	\draw (1) to[out=120, in=-120] (2) ;
	\draw (1) to[out=60, in=-60] (2) ;
	
	\draw (3) to (0.5,6) ;
	\draw (2) to[out=110, in = -90] (-0.5,6) ;
	
	\draw (1) to[out=-120, in = 90] (-0.4,0) ;
	\draw (1) to[out=-60, in = 90] (0.4,0) ;	
\end{tikzpicture} 
& $G_1(\Delta^0_{3,2},s)$ \\
\hline
\hline 
$0$ & $[2]^2$ & $[2]^2$ & $4$ & $4$ & $4$ & $2q + 16 + \dots$ \\
\hline 
$1$ & $\star$ & $\star$ & $\star$ & $2$ & $\star$ & $2q +14 + \dots$ \\
\hline 
\end{tabular}
\caption{Computation of $G_1(\Delta^0_{3,2},s)$.}
\label{table-square32g1}
\end{table}

\begin{landscape}
\vspace*{\fill}
\begin{table}[h!]
\renewcommand{\arraystretch}{1.5}
\centering
\begin{tabular}{|c||c|c|c|c|c|c|c|c||c|}
\hline 
$s$ &
\centering
\begin{tikzpicture}[scale=3/4]
	\floor (1) at (0,1.5) {} ;
	\floor (2) at (0,3) {} ;
    \floor (3) at (0,4.5) {} ;
	
	\draw (1) to node[left,pos=1/2]{\scriptsize $2$} (2) ;
	\draw (2) to node[left,pos=1/2]{\scriptsize $2$} (3) ;
	
	\draw (1) to[out=-120, in = 90] (-0.4,0) ;
	\draw (1) to[out=-60, in = 90] (0.4,0) ;
	
	\draw (3) to[out=120, in = -90] (-0.4,6) ;
	\draw (3) to[out=60, in = -90] (0.4,6) ;
\end{tikzpicture}
	 & 
\centering
\begin{tikzpicture}[scale=3/4]
	\floor (1) at (-0.5,1.5) {} ;
	\floor (2) at (0,3) {} ;
    \floor (3) at (0,4.5) {} ;
	
	\draw (1) to (2) ;
	\draw (2) to node[left,pos=1/2]{\scriptsize $2$} (3) ;
	
	\draw (1) to (-0.5,0) ;
	\draw (2) to[out=-70, in = 90] (0.5,0) ;
	
	\draw (3) to[out=120, in = -90] (-0.4,6) ;
	\draw (3) to[out=60, in = -90] (0.4,6) ;
\end{tikzpicture}
	 & 
\centering
\begin{tikzpicture}[scale=3/4]
	\floor (1) at (-0.5,1.5) {} ;
	\floor (2) at (-0.5,3) {} ;
    \floor (3) at (0,4.5) {} ;
	
	\draw (1) to (2) ;
	\draw (2) to (3) ;
	
	\draw (1) to (-0.5,0) ;
	\draw (3) to[out=-70, in = 90] (0.5,0) ;
	
	\draw (3) to[out=120, in = -90] (-0.4,6) ;
	\draw (3) to[out=60, in = -90] (0.4,6) ;		
\end{tikzpicture}
	 & 
\centering
\begin{tikzpicture}[scale=3/4]	
	\floor (1) at (0,1.5) {} ;
	\floor (2) at (0,3) {} ;
    \floor (3) at (-0.5,4.5) {} ;
	
	\draw (1) to node[left,pos=1/2]{\scriptsize $2$} (2) ;
	\draw (2) to (3) ;
	
	\draw (1) to[out=-120, in = 90] (-0.4,0) ;
	\draw (1) to[out=-60, in = 90] (0.4,0) ;
 
	\draw (2) to[out=70, in = -90] (0.5,6) ;
	
	\draw (3) to (-0.5,6) ;	
\end{tikzpicture} 
	 & 
\centering
\begin{tikzpicture}[scale=3/4]	
	\floor (1) at (0,1.5) {} ;
	\floor (2) at (-0.5,3) {} ;
    \floor (3) at (-0.5,4.5) {} ;
	
	\draw (1) to (2) ;
	\draw (2) to (3) ;
	
	\draw (1) to[out=-120, in = 90] (-0.4,0) ;
	\draw (1) to[out=-60, in = 90] (0.4,0) ;
 
	\draw (1) to[out=70, in = -90] (0.5,6) ;
	
	\draw (3) to (-0.5,6) ;	
\end{tikzpicture}
	 & 
\centering
\begin{tikzpicture}[scale=3/4]	
	\floor (1) at (-0.5,1.5) {} ;
	\floor (2) at (0,3) {} ;
    \floor (3) at (0.5,4.5) {} ;
	
	\draw (1) to (2) ;
	\draw (2) to (3) ;
 
	\draw (2) to[out=110, in = -90] (-0.5,6) ;
	\draw (2) to[out=-70, in = 90] (0.5,0) ;
	
	\draw (1) to (-0.5,0) ;	
	\draw (3) to (0.5,6) ;
\end{tikzpicture} 
	 & 
\centering
\begin{tikzpicture}[scale=3/4]	
	\floor (1) at (0,1.5) {} ;
	\floor (2) at (-1,3) {} ;
    \floor (3) at (1,3) {} ;
	
	\draw (1) to (2) ;
	\draw (1) to (3) ;
	
	\draw (1) to[out=-120, in = 90] (-0.4,0) ;
	\draw (1) to[out=-60, in = 90] (0.4,0) ;	
	
	\draw (2) to (-1,4.5) ;
	\draw (3) to (1,4.5) ;
\end{tikzpicture} 
	 & 
\centering
\begin{tikzpicture}[scale=3/4]	
	\floor (1) at (-1,1.5) {} ;
	\floor (2) at (1,1.5) {} ;
    \floor (3) at (0,3) {} ;
	
	\draw (3) to (2) ;
	\draw (3) to (1) ;
	
	\draw (1) to (-1,0) ;
	\draw (2) to (1,0) ;	
	
	\draw (3) to[out=120, in = 90] (-0.4,4.5) ;
	\draw (3) to[out=60, in = 90] (0.4,4.5) ;	
\end{tikzpicture} 
& $G_0(\Delta^0_{3,2},s)$ \\
\hline
\hline 
$0$ & $[2]^4$ & $4[2]^2$ & $6$ & $4[2]^2$ & $6$ & $16$ & $10$ & $10$ & $q^2 + 12q + 70 + \dots$ \\
\hline 
$1$ & $\star$ & $2[2]^2$ & $4$ & $\star$ & $\star$ & $8$ & $\star$ & $4$ & $q^2 + 10q + 50 + \dots$ \\
\hline 
$2$ & $[2]^2[2]_2$ & $\star$ & $2$ & $4[2]_2$  & $\star$ & $\star$ & $\star$ & $2$ & $q^2 + 8q + 34 + \dots$ \\
\hline 
$3$ & $([2]_2)^2$ & $2[2]_2$ & $\star$ & $\star$ & $4$ & $\star$ & $4$ & $\star$ & $q^2 + 6q + 22 + \dots$ \\
 \hline
$4$ & $\star$ & $\star$ & $\star$ & $2[2]_2$ & $2$ & $4$ & $2$ & $\star$ & $q^2 + 4q + 14 + \dots$ \\
 \hline
\end{tabular}
\caption{Computation of $G_0(\Delta^0_{3,2},s)$.}
\label{table-square32g0}
\end{table}
\vspace*{\fill}
\end{landscape}

\begin{ex} \label{ex-square23}
We compute $G_g(\Delta^0_{2,3},s)$ for $0\leq g \leq 2$. Tables \ref{table-square23g0} and \ref{table-square23g1} give
\begin{align*}
G_0(\Delta^0_{2,3},s) &= q^2 + (12-2s)q + (2s^2-22s+70) + \dots \\
G_1(\Delta^0_{2,3},s) &= 2q + (16-2s) + \dots \\
G_2(\Delta^0_{2,3},s) &= 1 .
\end{align*} 
\end{ex}

\begin{table}[h!]
\renewcommand{\arraystretch}{1.5}
\centering
\begin{tabular}{|c||c|c|c|c|c|c||c|}
\hline 
$s$ 
	 & 
\centering
\begin{tikzpicture}[scale=3/4]	
	\floor (1) at (0,1.5) {} ;
	\floor (2) at (0,3) {} ;
	
	\draw (1) to node[left,pos=1/2]{\scriptsize $3$} (2) ;
	
	\draw (1) to[out=-120, in = 90] (-0.4,0) ;
    \draw (1) to (0,0) ;
	\draw (1) to[out=-60, in = 90] (0.4,0) ;
	
	\draw (2) to[out=120, in = -90] (-0.4,4.5) ;
    \draw (2) to (0,4.5) ;
	\draw (2) to[out=60, in = -90] (0.4,4.5) ;
\end{tikzpicture}
	 & 
\centering
\begin{tikzpicture}[scale=3/4]
	\floor (1) at (0,1.5) {} ;
	\floor (2) at (0.5,3) {} ;
	
	\draw (1) to node[left,pos=1/2]{\scriptsize $2$} (2) ;
	
	\draw (1) to[out=-120, in = 90] (-0.4,0) ;
	\draw (1) to[out=-60, in = 90] (0.4,0) ;

    \draw (2) to[out=-60, in=90] (1.1,0) ;
	
	\draw (2) to[out=120, in = -90] (0.1,4.5) ;
    \draw (2) to (0.5,4.5) ;
	\draw (2) to[out=60, in = -90] (0.9,4.5) ;		
\end{tikzpicture}
	 & 
\centering
\begin{tikzpicture}[scale=3/4]	
	\floor (1) at (0,1.5) {} ;
	\floor (2) at (0.5,3) {} ;
	
	\draw (1) to (2) ;
	
	\draw (1) to (0,0) ;

    \draw (2) to[out=-70, in=90] (0.9,0) ;
    \draw (2) to[out=-50, in=90] (1.3,0) ;
	
	\draw (2) to[out=120, in = -90] (0.1,4.5) ;
    \draw (2) to (0.5,4.5) ;
	\draw (2) to[out=60, in = -90] (0.9,4.5) ;		
\end{tikzpicture} 
	 & 
\centering
\begin{tikzpicture}[scale=3/4]	
	\floor (1) at (0,1.5) {} ;
	\floor (2) at (0.5,3) {} ;
	
	\draw (1) to node[right,pos=1/2]{\scriptsize $2$} (2) ;
	
	\draw (1) to[out=-120, in = 90] (-0.4,0) ;
    \draw (1) to (0,0) ;
	\draw (1) to[out=-60, in = 90] (0.4,0) ;

    \draw (1) to[out=120, in=-90] (-0.6,4.5) ;
	
	\draw (2) to[out=120, in = -90] (0.1,4.5) ;
	\draw (2) to[out=60, in = -90] (0.9,4.5) ;		
\end{tikzpicture} 
	 & 
\centering
\begin{tikzpicture}[scale=3/4]	
	\floor (1) at (0,1.5) {} ;
	\floor (2) at (0.5,3) {} ;
	
	\draw (1) to (2) ;
	
	\draw (1) to[out=-120, in = 90] (-0.4,0) ;
    \draw (1) to (0,0) ;
	\draw (1) to[out=-60, in = 90] (0.4,0) ;

    \draw (1) to[out=130, in=-90] (-0.8,4.5) ;
    \draw (1) to[out=110, in=-90] (-0.4,4.5) ;
	
	\draw (2) to (0.5,4.5) ;
\end{tikzpicture}
	 & 
\centering
\begin{tikzpicture}[scale=3/4]	
	\floor (1) at (0,1.5) {} ;
	\floor (2) at (0.5,3) {} ;
	
	\draw (1) to (2) ;
	
	\draw (1) to[out=-120, in = 90] (-0.4,0) ;
	\draw (1) to[out=-60, in = 90] (0.4,0) ;

    \draw (1) to[out=120, in=-90] (-0.5,4.5) ;
    \draw (2) to[out=-60, in=90] (1,0) ;
	
	\draw (2) to[out=120, in = -90] (0.1,4.5) ;
	\draw (2) to[out=60, in = -90] (0.9,4.5) ;
\end{tikzpicture} 
& $G_0(\Delta^0_{2,3},s)$ \\
\hline
\hline 
$0$ & $[3]^2$ & $5[2]^2$ & $10$ & $5[2]^2$ & $10$ & $27$ & $q^2 + 12q + 70 + \dots$ \\
\hline 
$1$ & $\star$ & $3[2]^2$ & $4$ & $\star$ & $\star$ & $17$ & $q^2 + 10q + 50 + \dots$ \\
\hline 
$2$ & $\star$ & $[2]_2$ & $\star$  & $\star$ & $\star$ & $7$ & $q^2 +8q + 34 + \dots$ \\
\hline 
$3$ & $[3]_2$ & $\star$ & $\star$ & $2[3] + 3[2]_2$ & $\star$ & $5$ & $q^2 + 6q + 22 + \dots$ \\
\hline 
$4$ & $\star$ & $\star$ & $\star$ & $2[3] + [2]_2$ & $4$ & $3$ & $q^2 + 4q + 14 + \dots$ \\
\hline 
\end{tabular}
\caption{Computation of $G_0(\Delta^0_{2,3},s)$.}
\label{table-square23g0}
\end{table}
\begin{table}[h!]
\renewcommand{\arraystretch}{1.5}
\centering
\begin{tabular}{|c||c|c|c||c|}
\hline 
$s$ &
\centering
\begin{tikzpicture}[scale=3/4]
	\floor (1) at (0,1.5) {} ;
	\floor (2) at (0,3) {} ;
	
	\draw (1) to[out=60, in=-60] node[right,pos=1/2]{\scriptsize $2$} (2) ;
	\draw (1) to[out=120, in=-120] (2) ;
	
	\draw (2) to[out=120, in = -90] (-0.4,4.5) ;
    \draw (2) to (0,4.5) ;
	\draw (2) to[out=60, in = -90] (0.4,4.5) ;
	
	\draw (1) to[out=-120, in = 90] (-0.4,0) ;
	\draw (1) to (0,0) ;
	\draw (1) to[out=-60, in = 90] (0.4,0) ;
\end{tikzpicture}
	 & 
\centering
\begin{tikzpicture}[scale=3/4]	
	\floor (1) at (0,1.5) {} ;
	\floor (2) at (0.5,3) {} ;
	
	\draw (1) to[out=50, in=-90] (2) ;
	\draw (1) to[out=90, in=-130] (2) ;
	
	\draw (2) to[out=120, in=-90] (0.1,4.5) ;
    \draw (2) to[out=60, in=-90] (0.9,4.5) ;
	\draw (1) to[out=120, in=-90] (-0.6,4.5) ;
	
	\draw (1) to[out=-120, in = 90] (-0.4,0) ;
	\draw (1) to (0,0) ;
	\draw (1) to[out=-60, in = 90] (0.4,0) ;
\end{tikzpicture}
	 & 
\centering
\begin{tikzpicture}[scale=3/4]
	\floor (1) at (0,1.5) {} ;
	\floor (2) at (0.5,3) {} ;
	
	\draw (1) to[out=50, in=-90] (2) ;
	\draw (1) to[out=90, in=-130] (2) ;
	
	\draw (2) to[out=120, in=-90] (0.1,4.5) ;
    \draw (2) to (0.5,4.5) ;
	\draw (2) to[out=60, in=-90] (0.9,4.5) ;
	
	\draw (2) to[out=-60, in=90] (1.1,0) ;
	
	\draw (1) to[out=-120, in = 90] (-0.4,0) ;
	\draw (1) to[out=-60, in = 90] (0.4,0) ;		
\end{tikzpicture}
& $G_1(\Delta^0_{2,3},s)$ \\
\hline
\hline 
$0$ & $2[2]^2$ & $6$ & $6$ & $2q + 16 + \dots$ \\
\hline 
$1$ & $\star$ & $\star$ & $4$ & $2q + 14 + \dots$ \\
\hline 
\end{tabular}
\caption{Computation of $G_1(\Delta^0_{2,3},s)$.}
\label{table-square23g1}
\end{table}

\begin{ex} \label{ex-delta122}
We compute $G_g(\Delta^1_{2,2},s)$ for $0\leq g\leq 1$. Tables \ref{table-delta122g0} and \ref{table-delta122g1} give
\begin{align*}
    G_0(\Delta^1_{2,2},s) &= q^2 + (12-2s)q + (2s^2-22s+70) + \dots  \\
    G_1(\Delta^1_{2,2},s) &= 2q + (16-2s) + \dots \\
    G_2(\Delta^1_{2,2},s) &= 1.
\end{align*} 
\end{ex}

\begin{table}[h!]
\renewcommand{\arraystretch}{1.5}
\centering
\begin{tabular}{|c||c|c|c|c|c|c||c|}
\hline 
$s$ 
	 & 
\centering
\begin{tikzpicture}[scale=3/4]	
	\floor (1) at (0,1.5) {} ;
	\floor (2) at (0,3) {} ;
	
	\draw (1) to node[left,pos=1/2]{\scriptsize $3$} (2) ;
	
	\draw (1) to[out=-140, in = 90] (-0.6,0) ;
	\draw (1) to[out=-110, in = 90] (-0.2,0) ;
	\draw (1) to[out=-70, in = 90] (0.2,0) ;
	\draw (1) to[out=-40, in = 90] (0.6,0) ;
	
	\draw (2) to[out=120, in = -90] (-0.4,4.5) ;
	\draw (2) to[out=60, in = -90] (0.4,4.5) ;
\end{tikzpicture}
	 & 
\centering
\begin{tikzpicture}[scale=3/4]
	\floor (1) at (0,1.5) {} ;
	\floor (2) at (0.5,3) {} ;
	
	\draw (1) to node[left,pos=1/2]{\scriptsize $2$} (2) ;
	
	\draw (1) to[out=-120, in = 90] (-0.4,0) ;
    \draw (1) to (0,0) ;
	\draw (1) to[out=-60, in = 90] (0.4,0) ;

    \draw (2) to[out=-60, in=90] (1.1,0) ;
	
	\draw (2) to[out=120, in = -90] (0.1,4.5) ;
	\draw (2) to[out=60, in = -90] (0.9,4.5) ;		
\end{tikzpicture}
	 & 
\centering
\begin{tikzpicture}[scale=3/4]	
	\floor (1) at (0,1.5) {} ;
	\floor (2) at (0.5,3) {} ;
	
	\draw (1) to (2) ;
	
	\draw (1) to[out=-120, in = 90] (-0.4,0) ;
	\draw (1) to[out=-60, in = 90] (0.4,0) ;

    \draw (2) to[out=-70, in=90] (0.9,0) ;
    \draw (2) to[out=-50, in=90] (1.3,0) ;
	
	\draw (2) to[out=120, in = -90] (0.1,4.5) ;
	\draw (2) to[out=60, in = -90] (0.9,4.5) ;		
\end{tikzpicture} 
	 & 
\centering
\begin{tikzpicture}[scale=3/4]	
	\floor (1) at (0,1.5) {} ;
	\floor (2) at (0.5,3) {} ;
	
	\draw (1) to node[right,pos=1/2]{\scriptsize $2$} (2) ;
	
	\draw (1) to[out=-140, in = 90] (-0.6,0) ;
	\draw (1) to[out=-110, in = 90] (-0.2,0) ;
	\draw (1) to[out=-70, in = 90] (0.2,0) ;
	\draw (1) to[out=-40, in = 90] (0.6,0) ;

    \draw (1) to[out=120, in=-90] (-0.6,4.5) ;
	
	\draw (2) to (0.5,4.5) ;		
\end{tikzpicture} 
	 & 
\centering
\begin{tikzpicture}[scale=3/4]	
	\floor (1) at (0,1.5) {} ;
	\floor (2) at (0.5,3) {} ;
	
	\draw (1) to (2) ;
	
	\draw (1) to[out=-140, in = 90] (-0.6,0) ;
	\draw (1) to[out=-110, in = 90] (-0.2,0) ;
	\draw (1) to[out=-70, in = 90] (0.2,0) ;
	\draw (1) to[out=-40, in = 90] (0.6,0) ;

    \draw (1) to[out=130, in=-90] (-0.8,4.5) ;
    \draw (1) to[out=110, in=-90] (-0.4,4.5) ;
\end{tikzpicture}
	 & 
\centering
\begin{tikzpicture}[scale=3/4]	
	\floor (1) at (0,1.5) {} ;
	\floor (2) at (0.5,3) {} ;
	
	\draw (1) to (2) ;
	
	\draw (1) to[out=-120, in = 90] (-0.4,0) ;
    \draw (1) to (0,0) ;
	\draw (1) to[out=-60, in = 90] (0.4,0) ;

    \draw (1) to[out=120, in=-90] (-0.6,4.5) ;
    \draw (2) to[out=-60, in=90] (1.1,0) ;
	
	\draw (2) to (0.5,4.5) ;
\end{tikzpicture} 
& $G_0(\Delta^1_{2,2},s)$ \\
\hline
\hline 
$0$ & $[3]^2$ & $6[2]^2$ & $15$ & $4[2]^2$ & $6$ & $26$ & $q^2 + 12q + 70 + \dots$ \\
\hline 
$1$ & $\star$ & $4[2]^2$ & $7$ & $\star$ & $\star$ & $18$ & $q^2 + 10q + 50 + \dots$ \\
\hline 
$2$ & $\star$ & $2[2]_2$ & $3$  & $\star$ & $\star$ & $10$ & $q^2 +8q + 34 + \dots$ \\
\hline 
\end{tabular}
\caption{Computation of $G_0(\Delta^1_{2,2},s)$.}
\label{table-delta122g0}
\end{table}
\begin{table}[h!]
\renewcommand{\arraystretch}{1.5}
\centering
\begin{tabular}{|c||c|c|c||c|}
\hline 
$s$ &
\centering
\begin{tikzpicture}[scale=3/4]
	\floor (1) at (0,1.5) {} ;
	\floor (2) at (0,3) {} ;
	
	\draw (1) to[out=60, in=-60] node[right,pos=1/2]{\scriptsize $2$} (2) ;
	\draw (1) to[out=120, in=-120] (2) ;
	
	\draw (2) to[out=120, in = -90] (-0.4,4.5) ;
	\draw (2) to[out=60, in = -90] (0.4,4.5) ;
	
	\draw (1) to[out=-140, in = 90] (-0.6,0) ;
	\draw (1) to[out=-110, in = 90] (-0.2,0) ;
	\draw (1) to[out=-70, in = 90] (0.2,0) ;
	\draw (1) to[out=-40, in = 90] (0.6,0) ;
\end{tikzpicture}
	 & 
\centering
\begin{tikzpicture}[scale=3/4]	
	\floor (1) at (0,1.5) {} ;
	\floor (2) at (0.5,3) {} ;
	
	\draw (1) to[out=50, in=-90] (2) ;
	\draw (1) to[out=90, in=-130] (2) ;
	
	\draw (2) to (0.5,4.5) ;
	\draw (1) to[out=120, in=-90] (-0.6,4.5) ;
	
	\draw (1) to[out=-140, in = 90] (-0.6,0) ;
	\draw (1) to[out=-110, in = 90] (-0.2,0) ;
	\draw (1) to[out=-70, in = 90] (0.2,0) ;
	\draw (1) to[out=-40, in = 90] (0.6,0) ;
\end{tikzpicture}
	 & 
\centering
\begin{tikzpicture}[scale=3/4]
	\floor (1) at (0,1.5) {} ;
	\floor (2) at (0.5,3) {} ;
	
	\draw (1) to[out=50, in=-90] (2) ;
	\draw (1) to[out=90, in=-130] (2) ;
	
	\draw (2) to[out=120, in=-90] (0.1,4.5) ;
	\draw (2) to[out=60, in=-90] (0.9,4.5) ;
	
	\draw (2) to[out=-60, in=90] (1.1,0) ;
	
	\draw (1) to[out=-120, in = 90] (-0.4,0) ;
	\draw (1) to (0,0) ;
	\draw (1) to[out=-60, in = 90] (0.4,0) ;		
\end{tikzpicture}
& $G_1(\Delta^1_{2,2},s)$ \\
\hline
\hline 
$0$ & $2[2]^2$ & $5$ & $7$ & $2q + 16 + \dots$ \\
\hline 
$1$ & $\star$ & $\star$ & $5$ & $2q + 14 + \dots$ \\
\hline 
\end{tabular}
\caption{Computation of $G_1(\Delta^1_{2,2},s)$.}
\label{table-delta122g1}
\end{table}

\begin{ex} \label{ex-nabla122} 
Let $\nabla^1_{2,2}$ be the polygon obtained by applying a $\frac{\pi}{2}$-rotation to $\Delta^1_{2,2}$. In tables \ref{table-nabla122g0} and \ref{table-nabla122g1} the number inscribed in a vertex is its divergence. We obtain
\begin{align*}
    G_0(\nabla^1_{2,2},s) &= q^2 + (12-2s)q + (2s^2-22s+70) + \dots \\
    G_1(\nabla^1_{2,2},s) &= 2q + (16-2s) + \dots \\
    G_2(\nabla^1_{2,2},s) &= 1. 
\end{align*}
\end{ex}

\begin{landscape}
\vspace*{\fill}
\begin{table}[h!]
\renewcommand{\arraystretch}{1.5}
\centering
\hspace*{-3cm}\begin{tabular}{|c||c|c|c|c|c|c|c|c|c|c||c|}
\hline 
$s$ 
	 & 
\centering
\begin{tikzpicture}[scale=3/4]	
	\floor (1) at (0,1.5) {\scriptsize $0$} ;
	\floor (2) at (0,3) {\scriptsize $0$} ;
	\floor (3) at (0,4.5) {\scriptsize $1$} ;
	\floor (4) at (0,6) {\scriptsize $1$} ;
	
	\draw (1) to node[left,pos=1/2] {\scriptsize $2$} (2) ;
	\draw (2) to node[left,pos=1/2]{\scriptsize $2$} (3) ;
	\draw (3) to (4) ;
	
	\draw (1) to[out=-120, in = 90] (-0.4,0) ;
	\draw (1) to[out=-60, in = 90] (0.4,0) ;
\end{tikzpicture}
	 & 
\centering
\begin{tikzpicture}[scale=3/4]
	\floor (1) at (-0.5,1.5) {\scriptsize $0$} ;
	\floor (2) at (0,3) {\scriptsize $0$} ;
	\floor (3) at (0,4.5) {\scriptsize $1$} ;
	\floor (4) at (0,6) {\scriptsize $1$} ;
	
	\draw (1) to (2) ;
	\draw (2) to node[left,pos=1/2]{\scriptsize $2$} (3) ;
	\draw (3) to (4) ;

    \draw (2) to [out=-60, in=90] (0.6,0) ;
	\draw (1) to (-0.5,0) ;			
\end{tikzpicture}
	 & 
\centering
\begin{tikzpicture}[scale=3/4]	
	\floor (1) at (-0.5,1.5) {\scriptsize $0$} ;
	\floor (2) at (-0.5,3) {\scriptsize $0$} ;
	\floor (3) at (0,4.5) {\scriptsize $1$} ;
	\floor (4) at (0,6) {\scriptsize $1$} ;
	
	\draw (1) to (2) ;
	\draw (2) to (3) ;
	\draw (3) to (4) ;

    \draw (3) to [out=-70, in=90] (0.6,0) ;
	\draw (1) to (-0.5,0) ;		
\end{tikzpicture}
	 & 
\centering
\begin{tikzpicture}[scale=3/4]
	\floor (1) at (0,1.5) {\scriptsize $0$} ;
	\floor (2) at (0,3) {\scriptsize $1$} ;
	\floor (3) at (0,4.5) {\scriptsize $0$} ;
	\floor (4) at (0,6) {\scriptsize $1$} ;
	
	\draw (1) to node[left,pos=1/2] {\scriptsize $2$} (2) ;
	\draw (2) to (3) ;
	\draw (3) to (4) ;
	
	\draw (1) to[out=-120, in = 90] (-0.4,0) ;
	\draw (1) to[out=-60, in = 90] (0.4,0) ;
\end{tikzpicture} 
	 & 
\centering
\begin{tikzpicture}[scale=3/4]	
	\floor (1) at (-0.5,1.5) {\scriptsize $0$} ;
	\floor (2) at (0,3) {\scriptsize $1$} ;
	\floor (3) at (0,4.5) {\scriptsize $0$} ;
	\floor (4) at (0,6) {\scriptsize $1$} ;
	
	\draw (1) to (2) ;
	\draw (2) to (3) ;
	\draw (3) to (4) ;

    \draw (2) to [out=-60, in=90] (0.6,0) ;
	\draw (1) to (-0.5,0) ;	
\end{tikzpicture} 
	 & 
\centering
\begin{tikzpicture}[scale=3/4]	
	\floor (1) at (0,1.5) {\scriptsize $1$} ;
	\floor (2) at (0,3) {\scriptsize $0$} ;
	\floor (3) at (0,4.5) {\scriptsize $0$} ;
	\floor (4) at (0,6) {\scriptsize $1$} ;
	
	\draw (1) to (2) ;
	\draw (2) to (3) ;
	\draw (3) to (4) ;
	
	\draw (1) to[out=-120, in = 90] (-0.4,0) ;
	\draw (1) to[out=-60, in = 90] (0.4,0) ;
\end{tikzpicture}
	 & 
\centering
\begin{tikzpicture}[scale=3/4]	
	\floor (1) at (0,1.5) {\scriptsize $0$} ;
	\floor (2) at (0,3) {\scriptsize $0$} ;
	\floor (3) at (-1,4.5) {\scriptsize $1$} ;
	\floor (4) at (1,4.5) {\scriptsize $1$} ;
	
	\draw (1) to node[left,pos=1/2] {\scriptsize $2$} (2) ;
	\draw (2) to (3) ;
	\draw (2) to (4) ;
	
	\draw (1) to[out=-120, in = 90] (-0.4,0) ;
	\draw (1) to[out=-60, in = 90] (0.4,0) ;
\end{tikzpicture} 
 &
\centering
\begin{tikzpicture}[scale=3/4]	
	\floor (1) at (-0.5,1.5) {\scriptsize $0$} ;
	\floor (2) at (0,3) {\scriptsize $0$} ;
	\floor (3) at (-1,4.5) {\scriptsize $1$} ;
	\floor (4) at (1,4.5) {\scriptsize $1$} ;
	
	\draw (1) to (2) ;
	\draw (2) to (3) ;
	\draw (2) to (4) ;

    \draw (2) to[out=-60, in=90] (0.6,0) ;
	
	\draw (1) to (-0.5,0) ;
\end{tikzpicture} 
 &
\centering
\begin{tikzpicture}[scale=3/4]	
	\floor (1) at (0,1.5) {\scriptsize $0$} ;
	\floor (2) at (-1,3) {\scriptsize $0$} ;
	\floor (3) at (1,3) {\scriptsize $1$} ;
	\floor (4) at (-1,4.5) {\scriptsize $1$} ;
	
	\draw (1) to (2) ;
	\draw (1) to (3) ;
	\draw (2) to (4) ;
	
	\draw (1) to[out=-120, in = 90] (-0.4,0) ;
	\draw (1) to[out=-60, in = 90] (0.4,0) ;
\end{tikzpicture}
 &
\centering
\begin{tikzpicture}[scale=3/4]	
	\floor (1) at (-1,1.5) {\scriptsize $0$} ;
	\floor (2) at (1,1.5) {\scriptsize $0$} ;
	\floor (3) at (0,3) {\scriptsize $1$} ;
	\floor (4) at (0,4.5) {\scriptsize $1$} ;
	
	\draw (1) to (3) ;
	\draw (2) to (3) ;
	\draw (3) to (4) ;
	
	\draw (1) to (-1,0) ;
	\draw (2) to (1,0) ;
\end{tikzpicture}
& $G_0(\nabla^1_{2,2},s)$ \\
\hline
\hline 
$0$ & $[2]^4$ & $4[2]^2$ & $6$ & $[2]^2$ & $4$ & $1$ & $3[2]^2$ & $12$ & $15$ & $10$ & $q^2+12q+70$ \\
\hline 
$1$ & $\star$ & $2[2]^2$ & $4$ & $\star$ & $2$ & $\star$ & $\star$ & $6$ & $\star$ & $4$ & $q^2+10q+50$ \\
\hline 
$2$ & $[2]^2[2]_2$ & $\star$ & $2$  & $[2]_2$ & $\star$ & $\star$ & $3[2]_2$ & $\star$ & $\star$ & $2$ & $q^2+8q+34$ \\
\hline 
$3$ & $([2]_2)^2$ & $2[2]_2$ & $\star$ & $\star$ & $\star$ & $\star$ & $\star$ & $\star$ & $7$ & $\star$ & $q^2+6q+22$ \\
\hline
$4$ & $ \star$ & $\star$ & $\star$ & $\star$ & $\star$ & $\star$ & $[2]_2$ & $2$ & $3$ & $\star$ & $q^2+4q+14$ \\
\hline 
\end{tabular}
\caption{Computation of $G_0(\nabla^1_{2,2},s)$.}
\label{table-nabla122g0}
\end{table}
\vspace*{\fill}
\end{landscape}

\begin{table}[h!]
\renewcommand{\arraystretch}{1.5}
\centering
\begin{tabular}{|c||c|c|c|c|c|c||c|}
\hline 
$s$ 
	 & 
\centering
\begin{tikzpicture}[scale=3/4]	
	\floor (1) at (0,1.5) {\scriptsize $0$} ;
	\floor (2) at (0,3) {\scriptsize $0$} ;
	\floor (3) at (0,4.5) {\scriptsize $1$} ;
	\floor (4) at (0,6) {\scriptsize $1$} ;
	
	\draw (1) to[out=120, in=-120] (2) ;
	\draw (1) to[out=60, in=-60] (2) ;
	\draw (2) to node[left,pos=1/2]{\scriptsize $2$} (3) ;
	\draw (3) to (4) ;
	
	\draw (1) to[out=-120, in = 90] (-0.4,0) ;
	\draw (1) to[out=-60, in = 90] (0.4,0) ;
\end{tikzpicture}
	 & 
\centering
\begin{tikzpicture}[scale=3/4]
	\floor (1) at (0,1.5) {\scriptsize $0$} ;
	\floor (2) at (0,3) {\scriptsize $0$} ;
	\floor (3) at (0,4.5) {\scriptsize $1$} ;
	\floor (4) at (0,6) {\scriptsize $1$} ;
	
	\draw (2) to[out=120, in=-120] (3) ;
	\draw (2) to[out=60, in=-60] (3) ;
	\draw (1) to node[left,pos=1/2]{\scriptsize $2$} (2) ;
	\draw (3) to (4) ;
	
	\draw (1) to[out=-120, in = 90] (-0.4,0) ;
	\draw (1) to[out=-60, in = 90] (0.4,0) ;			
\end{tikzpicture}
	 & 
\centering
\begin{tikzpicture}[scale=3/4]	
	\floor (1) at (0,1.5) {\scriptsize $0$} ;
	\floor (2) at (-0.5,3) {\scriptsize $0$} ;
	\floor (3) at (0,4.5) {\scriptsize $1$} ;
	\floor (4) at (0,6) {\scriptsize $1$} ;
	
	\draw (1) to (2) ;
	\draw (1) to[out=60, in=-60] (3) ;
	\draw (2) to (3) ;
	\draw (3) to (4) ;
	
	\draw (1) to[out=-120, in = 90] (-0.4,0) ;
	\draw (1) to[out=-60, in = 90] (0.4,0) ;
\end{tikzpicture} 
	 & 
\centering
\begin{tikzpicture}[scale=3/4]	
	\floor (1) at (-0.5,1.5) {\scriptsize $0$} ;
	\floor (2) at (0,3) {\scriptsize $0$} ;
	\floor (3) at (0,4.5) {\scriptsize $1$} ;
	\floor (4) at (0,6) {\scriptsize $1$} ;
	
	\draw (1) to (2) ;
	\draw (2) to[out=120, in=-120] (3) ;
	\draw (2) to[out=60, in=-60] (3) ;
	\draw (3) to (4) ;
	
	\draw (1) to (-0.5,0) ;
	\draw (2) to[out=-70, in = 90] (0.5,0) ;
\end{tikzpicture} 
	 & 
\centering
\begin{tikzpicture}[scale=3/4]	
	\floor (1) at (0,1.5) {\scriptsize $0$} ;
	\floor (2) at (0,3) {\scriptsize $1$} ;
	\floor (3) at (0,4.5) {\scriptsize $0$} ;
	\floor (4) at (0,6) {\scriptsize $1$} ;
	
	\draw (1) to[out=120, in=-120] (2) ;
	\draw (1) to[out=60, in=-60] (2) ;
	\draw (2) to (3) ;
	\draw (3) to (4) ;
	
	\draw (1) to[out=-120, in = 90] (-0.4,0) ;
	\draw (1) to[out=-60, in = 90] (0.4,0) ;	
\end{tikzpicture}
	 & 
\centering
\begin{tikzpicture}[scale=3/4]	
	\floor (1) at (0,1.5) {\scriptsize $0$} ;
	\floor (2) at (0,3) {\scriptsize $0$} ;
	\floor (3) at (-1,4.5) {\scriptsize $1$} ;
	\floor (4) at (1,4.5) {\scriptsize $1$} ;
	
	\draw (1) to[out=120, in=-120] (2) ;
	\draw (1) to[out=60, in=-60] (2) ;
	\draw (2) to (3) ;
	\draw (2) to (4) ;
	
	\draw (1) to[out=-120, in = 90] (-0.4,0) ;
	\draw (1) to[out=-60, in = 90] (0.4,0) ;
\end{tikzpicture} 
& $G_1(\nabla^1_{2,2},s)$ \\
\hline
\hline 
$0$ & $2[2]^2$ & $[2]^2$ & $4$ & $4$ & $1$ & $3$ & $2q + 16 + \dots$ \\
\hline 
$1$ & $\star$ & $\star$ & $\star$ & $2$ & $\star$ & $\star$ & $2q + 14 + \dots$ \\
\hline 
\end{tabular}
\caption{Computation of $G_1(\nabla^1_{2,2},s)$.}
\label{table-nabla122g1}
\end{table}

\begin{ex} \label{ex-square22}
We compute $G_g(\Delta^0_{2,2},s)$ for $0\leq g \leq 1$. Table \ref{table-square22g0} gives
\begin{align*}
    G_0(\Delta^0_{2,2},s) &= q + (10-2s) + \dots \\
    G_1(\Delta^0_{2,2},s) &= 1. 
\end{align*} 
\end{ex}

\begin{table}[h!]
\renewcommand{\arraystretch}{1.5}
\centering
\begin{tabular}{|c||c|c|c||c|}
\hline 
$s$ &
\centering
\begin{tikzpicture}[scale=3/4]
	\floor (1) at (0,1.5) {} ;
	\floor (2) at (0,3) {} ;
	
	\draw (1) to node[left,pos=1/2]{\scriptsize $2$} (2) ;
	
	\draw (2) to[out=120, in = -90] (-0.4,4.5) ;
	\draw (2) to[out=60, in = -90] (0.4,4.5) ;
	
	\draw (1) to[out=-120, in = 90] (-0.4,0) ;
	\draw (1) to[out=-60, in = 90] (0.4,0) ;
\end{tikzpicture}
	 & 
\centering
\begin{tikzpicture}[scale=3/4]	
	\floor (1) at (0,1.5) {} ;
	\floor (2) at (0.5,3) {} ;
	
	\draw (1) to (2) ;
	
	\draw (2) to[out=120, in=-90] (0.1,4.5) ;
	\draw (2) to[out=60, in=-90] (0.9,4.5) ;
	
	\draw (1) to (0,0) ;
    \draw (2) to[out=-60, in=90] (1.1,0) ;
\end{tikzpicture}
	 & 
\centering
\begin{tikzpicture}[scale=3/4]
	\floor (1) at (0,1.5) {} ;
	\floor (2) at (0.5,3) {} ;
	
	\draw (1) to (2) ;
	
	\draw (2) to (0.5,4.5) ;
	\draw (1) to[out=120, in=-90] (-0.6,4.5) ;
	
	\draw (1) to[out=-120, in = 90] (-0.4,0) ;
	\draw (1) to[out=-60, in = 90] (0.4,0) ;
\end{tikzpicture}
& $G_0(\Delta^0_{2,2},s)$ \\
\hline
\hline 
$0$ & $[2]^2$ & $4$ & $4$ & $q + 10 + \dots$ \\
\hline 
$1$ & $\star$ & $2$ & $\star$ & $q + 8 + \dots$ \\
\hline 
\end{tabular}
\caption{Computation of $G_0(\Delta^0_{2,2},s)$.}
\label{table-square22g0}
%
%
%
%
\end{table}

\begin{ex} \label{ex-delta220}
We compute $G_g(\Delta^2_{2,0},s)$ for $0\leq g \leq 1$. Table \ref{table-delta220g0} gives
\begin{align*}
    G_0(\Delta^2_{2,0},s) &= q + (8-2s) + \dots \\
    G_1(\Delta^2_{2,0},s) &= 1. 
\end{align*} 
\end{ex}

\begin{table}[h!]
\renewcommand{\arraystretch}{1.5}
\centering
\begin{tabular}{|c||c|c||c|}
\hline 
$s$ &
\centering
\begin{tikzpicture}[scale=3/4]
	\floor (1) at (0,1.5) {} ;
	\floor (2) at (0,3) {} ;
	
	\draw (1) to node[left,pos=1/2]{\scriptsize $2$} (2) ;
	
	\draw (1) to[out=-140, in = 90] (-0.6,0) ;
	\draw (1) to[out=-110, in = 90] (-0.2,0) ;
	\draw (1) to[out=-70, in = 90] (0.2,0) ;
	\draw (1) to[out=-40, in = 90] (0.6,0) ;
\end{tikzpicture}
	 & 
\centering
\begin{tikzpicture}[scale=3/4]	
	\floor (1) at (0,1.5) {} ;
	\floor (2) at (0.5,3) {} ;
	
	\draw (1) to (2) ;
		
	\draw (1) to[out=-120, in = 90] (-0.4,0) ;
	\draw (1) to (0,0) ;
	\draw (1) to[out=-60, in = 90] (0.4,0) ;	
    \draw (2) to[out=-60, in=90] (1.1,0) ;
\end{tikzpicture}
& $G_0(\Delta^2_{2,0},s)$ \\
\hline
\hline 
$0$ & $[2]^2$ & $6$ & $q + 8 + \dots$ \\
\hline 
$1$ & $\star$ & $4$ & $q + 6 + \dots$ \\
\hline 
$2$ & $\star$ & $2$ & $q + 4 + \dots$ \\
\hline 
$3$  & $[2]_2$ & $\star$ & $q + 2 + \dots$ \\
 \hline
\end{tabular}
\caption{Computation of $G_0(\Delta^2_{2,0},s)$.}
\label{table-delta220g0}
\end{table}

\begin{ex} \label{ex-delta212}
We compute $G_g(\Delta^2_{1,2},s)$ for $0\leq g \leq 1$. For $g=0=\gmax(\Delta^2_{1,2})$ there is a unique marked floor diagram and it has multiplicity $1$. There is no diagram for $g=1$, hence
\begin{align*}
    G_0(\Delta^2_{1,2},s) &= 1, \\
    G_1(\Delta^2_{1,2},s) &= 0. 
\end{align*} 
\end{ex}

\begin{ex} \label{ex-delta221}
We compute $G_g(\Delta^2_{2,1},s)$ for $0\leq g \leq 2$. Tables \ref{table-delta221g0} and \ref{table-delta221g1} give
\begin{align*}
    G_0(\Delta^2_{2,1},s) &=  q^2 + (12-2s)q + (2s^2-22s+67) + \dots \\
    G_1(\Delta^2_{2,1},s) &= 2q + (16-2s) + \dots\\
    G_2(\Delta^2_{2,1},s) &= 1. 
\end{align*} 
\end{ex}

\begin{table}[h!]
\renewcommand{\arraystretch}{1.5}
\centering
\begin{tabular}{|c||c|c|c|c|c||c|}
\hline 
$s$ &
\centering
\begin{tikzpicture}[scale=3/4]
	\floor (1) at (0,1.5) {} ;
	\floor (2) at (0,3) {} ;
	
	\draw (1) to node[left,pos=1/2]{\scriptsize $3$} (2) ;
	
	\draw (2) to (0,4.5) ;
	
	\draw (1) to[out=-120, in = 90] (-0.4,0) ;
    \node at (0,0.5) {$\dots$} ;
	\node at (0,-0.4) {\large $\underbrace{ }_{5}$} ;
	\draw (1) to[out=-60, in = 90] (0.4,0) ;	
\end{tikzpicture}
	 & 
\centering
\begin{tikzpicture}[scale=3/4]	
	\floor (1) at (0,1.5) {} ;
	\floor (2) at (0.5,3) {} ;
	
	\draw (1) to node[right,pos=1/2]{\scriptsize $2$} (2) ;
	
	\draw (1) to[out=110, in=-90] (-0.5,4.5) ;
	
	\draw (1) to[out=-120, in = 90] (-0.4,0) ;
    \node at (0,0.5) {$\dots$} ;
	\node at (0,-0.4) {\large $\underbrace{ }_{5}$} ;
	\draw (1) to[out=-60, in = 90] (0.4,0) ;
\end{tikzpicture}
	 & 
\centering
\begin{tikzpicture}[scale=3/4]
	\floor (1) at (0,1.5) {} ;
	\floor (2) at (0.5,3) {} ;
	
	\draw (1) to node[left,pos=1/2] {\scriptsize $2$} (2) ;
	\draw (2) to (0.5,4.5) ;
	
	\draw (1) to[out=-140, in = 90] (-0.6,0) ;
	\draw (1) to[out=-110, in = 90] (-0.2,0) ;
	\draw (1) to[out=-70, in = 90] (0.2,0) ;
	\draw (1) to[out=-40, in = 90] (0.6,0) ;

    \draw (2) to[out=-60, in=90] (1.1,0) ;
\end{tikzpicture}
	 & 
\centering
\begin{tikzpicture}[scale=3/4]	
	\floor (1) at (0,1.5) {} ;
	\floor (2) at (0.5,3) {} ;
	
	\draw (1) to (2) ;
    \draw (2) to (0.5,4.5) ;
		
	\draw (1) to[out=-120, in = 90] (-0.4,0) ;
	\draw (1) to (0,0) ;
	\draw (1) to[out=-60, in = 90] (0.4,0) ;	
 
	\draw (2) to[out=-70, in = 90] (1,0) ;
	\draw (2) to[out=-50, in = 90] (1.4,0) ;		
\end{tikzpicture}
	 & 
\centering
\begin{tikzpicture}[scale=3/4]	
	\floor (1) at (0,1.5) {} ;
	\floor (2) at (0.5,3) {} ;
		
	\draw (1) to (2) ;
	\draw (1) to[out=110, in=-90] (-0.5,4.5) ;

    \draw (2) to[out=-60, in=90] (1.1,0) ;
	
	\draw (1) to[out=-140, in = 90] (-0.6,0) ;
	\draw (1) to[out=-110, in = 90] (-0.2,0) ;
	\draw (1) to[out=-70, in = 90] (0.2,0) ;
	\draw (1) to[out=-40, in = 90] (0.6,0) ;
\end{tikzpicture} 
& $G_0(\Delta^2_{2,1},s)$ \\
\hline
\hline 
$0$ & $[3]^2$ & $3[2]^2$ & $7[2]^2$ & $21$ & $23$ & $q^2 + 12q + 67 + \dots$ \\
\hline 
$1$ & $\star$ & $\star$ & $5[2]^2$ & $11$ & $17$ & $q^2 + 10q + 47 + \dots$ \\
\hline 
$2$ & $\star$ & $\star$ & $3[2]^2$ & $5$ & $11$ & $q^2 + 8q + 31 + \dots$ \\
\hline 
\end{tabular}
\caption{Computation of $G_0(\Delta^2_{2,1},s)$.}
\label{table-delta221g0}
\end{table}
\begin{table}[h!]
\renewcommand{\arraystretch}{1.5}
\centering
\begin{tabular}{|c||c|c|c||c|}
\hline 
$s$ &
\centering
\begin{tikzpicture}[scale=3/4]
	\floor (1) at (0,1.5) {} ;
	\floor (2) at (0,3) {} ;
	
	\draw (1) to[out=60, in=-60] (2) ;
	\draw (1) to[out=120, in=-120] node[left,pos=1/2]{\scriptsize $2$} (2) ;
	
    \draw (2) to (0,4.5) ;
	
	\draw (1) to[out=-120, in = 90] (-0.4,0) ;
    \node at (0,0.5) {$\dots$} ;
	\node at (0,-0.4) {\large $\underbrace{ }_{5}$} ;
	\draw (1) to[out=-60, in = 90] (0.4,0) ;
\end{tikzpicture}
	 & 
\centering
\begin{tikzpicture}[scale=3/4]	
	\floor (1) at (0,1.5) {} ;
	\floor (2) at (0.5,3) {} ;
	
	\draw (1) to[out=50, in=-90] (2) ;
	\draw (1) to[out=90, in=-130] (2) ;
	
	\draw (1) to[out=120, in=-90] (-0.6,4.5) ;
	
	\draw (1) to[out=-120, in = 90] (-0.4,0) ;
    \node at (0,0.5) {$\dots$} ;
	\node at (0,-0.4) {\large $\underbrace{ }_{5}$} ;
	\draw (1) to[out=-60, in = 90] (0.4,0) ;
\end{tikzpicture}
	 & 
\centering
\begin{tikzpicture}[scale=3/4]
	\floor (1) at (0,1.5) {} ;
	\floor (2) at (0.5,3) {} ;
	
	\draw (1) to[out=50, in=-90] (2) ;
	\draw (1) to[out=90, in=-130] (2) ;
	
    \draw (2) to (0.5,4.5) ;
	
	\draw (2) to[out=-60, in=90] (1.1,0) ;
	
	\draw (1) to[out=-140, in = 90] (-0.6,0) ;
	\draw (1) to[out=-110, in = 90] (-0.2,0) ;
	\draw (1) to[out=-70, in = 90] (0.2,0) ;
	\draw (1) to[out=-40, in = 90] (0.6,0) ;	
\end{tikzpicture}
& $G_1(\Delta^2_{2,1},s)$ \\
\hline
\hline 
$0$ & $2[2]^2$ & $4$ & $8$ & $2q + 16 + \dots$ \\
\hline 
$1$ & $\star$ & $\star$ & $6$ & $2q + 14 + \dots$ \\
\hline 
\end{tabular}
\caption{Computation of $G_1(\Delta^2_{2,1},s)$.}
\label{table-delta221g1}
\end{table}


\begin{ex} \label{ex-delta213}
We compute $G_g(\Delta^2_{1,3},s)$ for $0\leq g \leq 1$. For $g=0=\gmax(\Delta^2_{1,3})$ there is a unique marked floor diagram and it has multiplicity $1$. There is no diagram for $g=1$, hence
\begin{align*}
    G_0(\Delta^2_{1,3},s) &= 1, \\
    G_1(\Delta^2_{1,3},s) &= 0. 
\end{align*} 
\end{ex}

\begin{ex} \label{ex-square33}
We compute $G_3(\Delta^0_{3,3},s)$. 
Table \ref{table-square33g3} gives
\[ G_3(\Delta^0_{3,3},s) = 4q + (26-2s) + \dots  \]
\end{ex}

\begin{table}[h!]
\renewcommand{\arraystretch}{1.5}
\centering
\begin{tabular}{|c||c|c|c|c|c||c|}
\hline 
$s$ &
\centering
\begin{tikzpicture}[scale=3/4]
	\floor (1) at (0,1.5) {} ;
	\floor (2) at (0,3) {} ;
	\floor (3) at (0,4.5) {} ;
	
	\draw (1) to[out=120, in=-120] node[left,pos=1/2]{\scriptsize $2$} (2) ;
    \draw (1) to[out=60, in=-60] (2) ;
	
	\draw (2) to[out=120, in=-120] (3) ;
    \draw (2) to (3) ;
	\draw (2) to[out=60, in=-60] (3) ;
	
	\draw (1) to[out=-120, in = 90] (-0.4,0) ;
    \draw (1) to (0,0) ;
	\draw (1) to[out=-60, in = 90] (0.4,0) ;
	
	\draw (3) to[out=120, in = -90] (-0.4,6) ;
    \draw (3) to (0,6) ;
	\draw (3) to[out=60, in = -90] (0.4,6) ;
\end{tikzpicture}
	 & 
\centering
\begin{tikzpicture}[scale=3/4]	
	\floor (1) at (0,1.5) {} ;
	\floor (2) at (0,3) {} ;
	\floor (3) at (0,4.5) {} ;
	
	\draw (1) to[out=120, in=-120] (2) ;
    \draw (1) to (2) ;
    \draw (1) to[out=60, in=-60] (2) ;
	
	\draw (2) to[out=120, in=-120] node[left,pos=1/2]{\scriptsize $2$} (3) ;
	\draw (2) to[out=60, in=-60] (3) ;
	
	\draw (1) to[out=-120, in = 90] (-0.4,0) ;
    \draw (1) to (0,0) ;
	\draw (1) to[out=-60, in = 90] (0.4,0) ;
	
	\draw (3) to[out=120, in = -90] (-0.4,6) ;
    \draw (3) to (0,6) ;
	\draw (3) to[out=60, in = -90] (0.4,6) ;
\end{tikzpicture}
	 & 
\centering
\begin{tikzpicture}[scale=3/4]	
	\floor (1) at (-0.5,1.5) {} ;
	\floor (2) at (0,3) {} ;
	\floor (3) at (0,4.5) {} ;
	
	\draw (1) to[out=50, in=-90] (2) ;
	\draw (1) to[out=90, in=-130] (2) ;
	
	\draw (2) to[out=120, in=-120] (3) ;
    \draw (2) to (3) ;
	\draw (2) to[out=60, in=-60] (3) ;
 
	\draw (2) to[out=-60, in = 90] (0.7,0) ;

	\draw (1) to[out=-120, in = 90] (-0.9,0) ;
	\draw (1) to[out=-60, in = 90] (-0.1,0) ;
	
	\draw (3) to[out=120, in = -90] (-0.4,6) ;
    \draw (3) to (0,6) ;
	\draw (3) to[out=60, in = -90] (0.4,6) ;	
\end{tikzpicture}
	 & 
\centering
\begin{tikzpicture}[scale=3/4]	
	\floor (1) at (0,1.5) {} ;
	\floor (2) at (0,3) {} ;
	\floor (3) at (0.5,4.5) {} ;
	
	\draw (2) to[out=50, in=-90] (3) ;
	\draw (2) to[out=90, in=-130] (3) ;
	
	\draw (1) to[out=120, in=-120] (2) ;
    \draw (1) to (2) ;
    \draw (1) to[out=60, in=-60] (2) ;
	
	\draw (3) to[out=120, in=-90] (0.1,6) ;
	\draw (3) to[out=60, in=-90] (0.9,6) ;
 
	\draw (2) to[out=120, in = -90] (-0.7,6) ;
	
	\draw (1) to[out=-120, in = 90] (-0.4,0) ;
    \draw (1) to (0,0) ;
	\draw (1) to[out=-60, in = 90] (0.4,0) ;
\end{tikzpicture} 
	 & 
\centering
\begin{tikzpicture}[scale=3/4]	
	\floor (1) at (0,1.5) {} ;
	\floor (2) at (0.5,3) {} ;
	\floor (3) at (0,4.5) {} ;
	
	\draw (1) to[out=50, in=-90] (2) ;
	\draw (1) to[out=90, in=-130] (2) ;
	
	\draw (2) to[out=130, in=-90] (3) ;
	\draw (2) to[out=90, in=-50] (3) ;

    \draw (1) to[out=120, in=-120] (3) ;
	
	\draw (1) to[out=-120, in = 90] (-0.4,0) ;
    \draw (1) to (0,0) ;
	\draw (1) to[out=-60, in = 90] (0.4,0) ;
	
	\draw (3) to[out=120, in = -90] (-0.4,6) ;
    \draw (3) to (0,6) ;
	\draw (3) to[out=60, in = -90] (0.4,6) ;
\end{tikzpicture} 
& $G_3(\Delta^0_{3,3},s)$ \\
\hline
\hline 
$0$ & $2[2]^2$ & $2[2]^2$ & $6$ & $6$ & $6$ & $4q + 26 + \dots$ \\
\hline 
$1$ & $\star$ & $\star$ & $4$ & $\star$ & $\star$ & $4q + 24 + \dots$ \\
\hline 
\end{tabular}
\caption{Computation of $G_3(\Delta^0_{3,3},s)$.}
\label{table-square33g3}
\end{table}

\begin{ex} \label{ex-delta230}
We compute $G_3(\Delta^2_{3,0},s)$. Table \ref{table-delta230g3} gives
\[  G_3(\Delta^2_{3,0},s) = 4q + (24-2s) + \dots \] 
\end{ex}

\begin{table}[h!]
\renewcommand{\arraystretch}{1.5}
\centering
\begin{tabular}{|c||c|c|c|c||c|}
\hline 
$s$ &
\centering
\begin{tikzpicture}[scale=3/4]
	\floor (1) at (0,1.5) {} ;
	\floor (2) at (0,3) {} ;
    \floor (3) at (0,4.5) {} ;
	
	\draw (2) to node[left,pos=1/2]{\scriptsize $2$} (3) ;
	
	\draw (1) to[out=140, in=-140] (2) ;
	\draw (1) to[out=110, in=-110] (2) ;
	\draw (1) to[out=70, in=-70] (2) ;
	\draw (1) to[out=40, in=-40] (2) ;
 
	\draw (1) to[out=-120, in = 90] (-0.4,0) ;
    \node at (0,0.5) {$\dots$} ;
	\node at (0,-0.4) {\large $\underbrace{ }_{6}$} ;
	\draw (1) to[out=-60, in = 90] (0.4,0) ;
\end{tikzpicture}
	 & 
\centering
\begin{tikzpicture}[scale=3/4]	
	\floor (1) at (0,1.5) {} ;
	\floor (2) at (0,3) {} ;
    \floor (3) at (0,4.5) {} ;
	
	\draw (2) to[out=120, in=-120] (3) ;
	\draw (2) to[out=60, in=-60] (3) ;
	
	\draw (1) to[out=120, in=-120] node[left,pos=1/2]{\scriptsize $2$} (2) ;
	\draw (1) to (2) ;
	\draw (1) to[out=60, in=-60] (2) ;
 
	\draw (1) to[out=-120, in = 90] (-0.4,0) ;
    \node at (0,0.5) {$\dots$} ;
	\node at (0,-0.4) {\large $\underbrace{ }_{6}$} ;
	\draw (1) to[out=-60, in = 90] (0.4,0) ;
\end{tikzpicture}
	 & 
\centering
\begin{tikzpicture}[scale=3/4]
	\floor (1) at (0,1.5) {} ;
	\floor (2) at (0.5,3) {} ;
    \floor (3) at (0.5,4.5) {} ;
	
	\draw (2) to[out=120, in=-120] (3) ;
	\draw (2) to[out=60, in=-60] (3) ;
	
	\draw (1) to[out=40, in=-70] (2) ;
	\draw (1) to[out=65, in=-100] (2) ;
	\draw (1) to[out=90, in=-130] (2) ;
	
	\draw (2) to[out=-50, in=90] (1.3,0) ;
 
	\draw (1) to[out=-120, in = 90] (-0.4,0) ;
    \node at (0,0.5) {$\dots$} ;
	\node at (0,-0.4) {\large $\underbrace{ }_{5}$} ;
	\draw (1) to[out=-60, in = 90] (0.4,0) ;
\end{tikzpicture}
	 & 
\centering
\begin{tikzpicture}[scale=3/4]	
	\floor (1) at (0,1.5) {} ;
	\floor (2) at (0.5,3) {} ;
    \floor (3) at (0,4.5) {} ;
	
	\draw (2) to (3) ;
	
	\draw (1) to[out=40, in=-70] (2) ;
	\draw (1) to[out=65, in=-100] (2) ;
	\draw (1) to[out=90, in=-130] (2) ;

    \draw (1) to[out=120, in=-120] (3) ;
 
	\draw (1) to[out=-120, in = 90] (-0.4,0) ;
    \node at (0,0.5) {$\dots$} ;
	\node at (0,-0.4) {\large $\underbrace{ }_{6}$} ;
	\draw (1) to[out=-60, in = 90] (0.4,0) ;
\end{tikzpicture}
& $G_3(\Delta^2_{3,0},s)$ \\
\hline
\hline 
$0$ & $[2]^2$ & $3[2]^2$ & $10$ & $6$ & $4q + 24 + \dots$ \\
\hline 
$1$ & $\star$ & $\star$ & $8$ & $\star$ & $4q + 22 + \dots$ \\
\hline 
\end{tabular}
\caption{Computation of $G_3(\Delta^2_{3,0},s)$.}
\label{table-delta230g3}
\end{table}

\begin{ex} \label{ex-delta222}
We compute $G_3(\Delta^2_{2,2},s)$. For $g=3=\gmax(\Delta^2_{2,2})$ there is a unique marked floor diagram and it has multiplicity $1$, hence
\[ G_3(\Delta^2_{2,2},s) = 1.\]
\end{ex}

\begin{ex} \label{ex-delta214}
We compute $G_3(\Delta^2_{1,4},s)$. Since $\gmax(\Delta^2_{1,4})=0$, one has
\[  G_3(\Delta^2_{1,4},s) = 0.\]
\end{ex}

\begin{ex}
	As a last example, we perform the computation for $\Delta^1_{d,0}$ and with genus $g = \gmax(\Delta^1_{d,0})-1 = \frac{(d-1)(d-2)}{2}-1$, with $d\geq3$. We take $s \leq d/2$ and use the pairing $S = \{\{1,2\},\dots,\{2s-1,2s\}\}$. Note that thanks to theorem \ref{theo-GSinv-poly-s} we could make the computation only for $s\leq1$. Let $v_d \prec \dots \prec v_2 \prec v_1$ be the vertices of the unique floor diagram with genus $\gmax(\Delta^1_{d,0})$. There are two possible ways to construct a diagram of genus $g$.	
	\begin{itemize}
		\item One can merge two bounded edges into an edge of weight $2$, see figure \ref{fig-diag-CP2-gmax-1-a}.
		
		\item One can choose a vertex $v_i$ for $2 \leq i \leq d$, delete an edge below and above $v_i$, then add an edge adjacent to $v_{i-1}$ and $v_{i+1}$, see figures \ref{fig-diag-CP2-gmax-1-b}, \ref{fig-diag-CP2-gmax-1-c} and \ref{fig-diag-CP2-gmax-1-d} (in figure \ref{fig-diag-CP2-gmax-1-d}, one understands $v_{d+1}$ as a vertex at infinity, hence the added edge is infinite).
	\end{itemize}
	
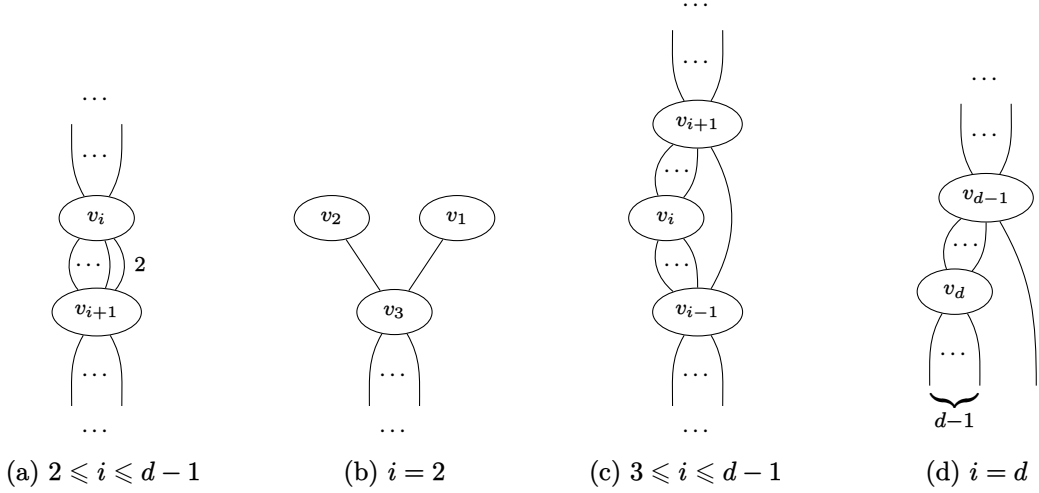
\begin{figure}[h!]
\begin{subfigure}[t]{0.24\textwidth}
	\centering
	\begin{tikzpicture}[scale=5/6]
	\floor (1) at (0,1.5) {\scriptsize $v_{i+1}$} ;
	\floor (2) at (0,3) {\scriptsize $v_i$} ;
	
	\draw (1) to[out=130, in=-130] (2) ;
    \node at (-0.1,2.25) {\scriptsize $\dots$} ;
	\draw (1) to[out=70, in=-70] (2) ;
	\draw (1) to[out=50, in=-50] node[right,pos=1/2] {\scriptsize $2$} (2) ;
	
	\draw (1) to[out=-120, in=90] (-0.4,0) ;
	\draw (1) to[out=-60, in=90] (0.4,0) ;
    \node at (0,0.5) {\scriptsize $\dots$} ;
    \node at (0,-0.4) {\scriptsize $\dots$} ;
	
	\draw (2) to[out=120, in=-90] (-0.4,4.5) ;
	\draw (2) to[out=60, in=-90] (0.4,4.5) ;
    \node at (0,4) {\scriptsize $\dots$} ;
    \node at (0,4.9) {\scriptsize $\dots$} ;
	
	\end{tikzpicture}
	\caption{$2\leq i \leq d-1$}
	\label{fig-diag-CP2-gmax-1-a}
\end{subfigure}
\begin{subfigure}[t]{0.24\textwidth}
	\centering
	\begin{tikzpicture}[scale=5/6]
	\floor (0) at (0,1.5) {\scriptsize $v_3$} ;
	\floor (1) at (-1,3) {\scriptsize $v_2$} ;
	\floor (2) at (1,3) {\scriptsize $v_1$} ;
	
	\draw (0) to (1) ;
	\draw (0) to (2) ;	
	
	\draw (0) to[out=-120, in=90] (-0.4,0) ;
	\draw (0) to[out=-60, in=90] (0.4,0) ;
    \node at (0,0.5) {\scriptsize $\dots$} ;
    \node at (0,-0.4) {\scriptsize $\dots$} ;	
	\end{tikzpicture}
	\caption{$i=2$}
	\label{fig-diag-CP2-gmax-1-b}
\end{subfigure}
\begin{subfigure}[t]{0.24\textwidth}
	\centering
	\begin{tikzpicture}[scale=5/6]
	\floor (0) at (0,1.5) {\scriptsize $v_{i-1}$} ;
	\floor (1) at (-0.5,3) {\scriptsize $v_i$} ;
	\floor (2) at (0,4.5) {\scriptsize $v_{i+1}$} ;

	\draw (0) to[out=140, in=-110] (1) ;
	\draw (0) to[out=90, in=-50] (1) ;
    \node at (-0.3,2.25) {\scriptsize $\dots$} ;

	\draw (1) to[out=110, in=-140] (2) ;
	\draw (1) to[out=50, in=-90] (2) ;
    \node at (-0.3,3.75) {\scriptsize $\dots$} ;
	
    \draw (0) to[out=60, in=-60] (2) ;
	
	\draw (0) to[out=-120, in=90] (-0.4,0) ;
	\draw (0) to[out=-60, in=90] (0.4,0) ;
    \node at (0,0.5) {\scriptsize $\dots$} ;
    \node at (0,-0.4) {\scriptsize $\dots$} ;
	
	\draw (2) to[out=120, in=-90] (-0.4,6) ;
	\draw (2) to[out=60, in=-90] (0.4,6) ;
    \node at (0,5.5) {\scriptsize $\dots$} ;
    \node at (0,6.4) {\scriptsize $\dots$} ;	
	\end{tikzpicture}
	\caption{$3\leq i \leq d-1$}
	\label{fig-diag-CP2-gmax-1-c}
\end{subfigure}
\begin{subfigure}[t]{0.24\textwidth}
	\centering
	\begin{tikzpicture}[scale=5/6]
	\floor (0) at (-0.5,1.5) {\scriptsize $v_d$} ;
	\floor (1) at (0,3) {\scriptsize $v_{d-1}$} ;

	\draw (0) to[out=110, in=-140] (1) ;
	\draw (0) to[out=50, in=-90] (1) ;
    \node at (-0.3,2.25) {\scriptsize $\dots$} ;

    \draw (1) to[out=-60, in=90] (0.8,0) ; ;
	
	\draw (0) to[out=-120, in=90] (-0.9,0) ;
	\draw (0) to[out=-60, in=90] (-0.1,0) ;
    \node at (-0.5,0.5) {\scriptsize $\dots$} ;
	\node at (-0.5,-0.4) {\large $\underbrace{ }_{d-1}$} ;	
	
	\draw (1) to[out=120, in=-90] (-0.4,4.5) ;
	\draw (1) to[out=60, in=-90] (0.4,4.5) ;
    \node at (0,4) {\scriptsize $\dots$} ;
    \node at (0,4.9) {\scriptsize $\dots$} ;	
	\end{tikzpicture}
	\caption{$i=d$}
	\label{fig-diag-CP2-gmax-1-d}
\end{subfigure}
\caption{The floor diagrams with Newton polygon $\Delta^1_{d,0}$ and genus $\gmax(\Delta^1_{d,0})-1$.}
\label{fig-diag-CP2-gmax-1}
\end{figure}	
	
	In case (a), the $S$-multiplicity is $[2]^2$. If the bounded edge of weight $2$ is adjacent to $v_i$ and $v_{i+1}$, then there are $i-1$ markings compatible with $S$. One can choose $2\leq i \leq d-1$, hence the case (a) contributes
	\[ [2]^2 \dsum_{i=2}^{d-1} (i-1) = [2]^2 \dfrac{(d-1)(d-2)}{2} \]
	to $G_g(\Delta^1_{d,0},s)$.
	
	In cases (b), (c) and (d) the $S$-multiplicity is $1$. In case (b) the number of markings is $3$. In case (c) the number of markings is $2i-1$. Last, in case (d) the number of compatible markings is $2d-1-2s$. Hence the contribution of cases (b), (c) and (d) to $G_g(\Delta^1_{d,0},s)$ is
	\[ 3 + \dsum_{i=3}^{d-1}(2i-1) + 2d-1-2s = d^2-1-2s .\]
	
	In the end one has
	\begin{align*}
	G_{\gmax(\Delta^1_{d,0})-1}(\Delta^1_{d,0},s)
	  &= \dfrac{(d-1)(d-2)}{2} q + (2d^2-3d+1-2s) + \dots \\
	  &= \binom{\gmax(\Delta^1_{d,0})}{\gmax(\Delta^1_{d,0})-1} q + (2d^2-3d+1-2s) + \dots
	\end{align*}	
In particular, for $d=3$ we get $G_0(\Delta_3,s) = q + (10-2s) + q^{-1}$ and we recover example \ref{ex-calcul-G0-Delta3}.	
\end{ex}

\subsection{Observations and conjecture}

From the calculations of the previous section \ref{subsec-ex-combi-GS-inv}, one can make several observations leading to few conjectures.

\subsubsection{Invariance under lattice preserving transformation}

Recall a lattice preserving transformation is an element of the affine group of $\R^2$ for which the lattice $\Z^2$ is invariant. 
In the previous section, several polygons for which we performed calculations are linked by a lattice preserving transformation.  First, examples \ref{ex-square32} and \ref{ex-square23} show that for $0\leq g \leq 2$ one has
\[ G_g(\Delta^0_{2,3},s) = G_g(\Delta^0_{3,2},s) .\]
Second, examples \ref{ex-delta122} and \ref{ex-nabla122} show that for $0\leq g \leq 2$ one has
\[ G_g(\Delta^1_{2,2},s) = G_g(\nabla^1_{2,2},s) .\]
Although not detailed in this paper, one can check for instance that for $0\leq g \leq 2$ one also has
\[ G_g(\Delta^2_{2,1},s) = G_g(A\Delta^2_{2,1},s) \]
where $A$ is the matrix $A = \begin{pmatrix} 1 & 0 \\ 1 & 1 \end{pmatrix}$, and for $g=3$ one has
\[ G_3(\Delta^2_{3,0},s) = G_3(A\Delta^2_{3,0},s) . \]
All these observations lead to the following conjecture.

\begin{conj}
    Let $\Delta$ and $\Delta'$ be two $h$-transverse polygons. If there exists a lattice preserving transformation $f$ such that $f(\Delta) = \Delta'$, then for any $g\in \{0,\dots,\gmax(\Delta)\}$ and $s\in\{0,\dots,\smax(\Delta,g)\}$ one has
    \[ G_g(\Delta,s) = G_g(\Delta',s) .\]
\end{conj}

By the results of \cite{blomme24asymptotic} this conjecture is asymptotically true in genus $1$, if $\Delta$ and $\Delta'$ are moreover non-singular and horizontal. Indeed, if the integral lengths of the sides of $\Delta$ are large enough, we know that the coefficients of small codegree of $G_1(\Delta,s)$ are given by polynomials which only depend on $y(\Delta)$, $\chi(\Delta)$ and $\gmax(\Delta)$, and similarly for $\Delta'$. Since the triplet $(y,\chi,\gmax)$ is the same for $\Delta$ and $\Delta'$, then the coefficients of small codegrees of $G_1(\Delta,s)$ and $G_1(\Delta',s)$ are the same.

\subsubsection{Abramovich-Bertram formula}

We already know by \cite{bousseau_refined_2021} that Block-Göttsche refined invariants, \ie $G_g(\Delta,0)$, and genus $0$ Göttsche-Schroeter invariants, \ie $G_0(\Delta,s)$, satisfy the Abramovich-Bertram formula. One can wonder if this formula also holds for $g$ and $s$ both non-zero. Some examples of the previous section would plea in favor of a positive answer.
From examples \ref{ex-square23}, \ref{ex-delta221} and \ref{ex-delta213} we observe that for $0\leq g\leq 2$ one has
\[ G_g(\Delta^0_{2,3},s) = G_g(\Delta^2_{2,1},s) + 3\times G_g(\Delta^2_{1,3},s). \]
From examples \ref{ex-square22}, \ref{ex-delta220} and \ref{ex-delta212} we observe that for $0\leq g\leq 2$ one has
\[ G_g(\Delta^0_{2,2},s) = G_g(\Delta^2_{2,0},s) + 2\times G_g(\Delta^2_{1,2},s). \]
From examples \ref{ex-square33}, \ref{ex-delta230}, \ref{ex-delta222} and \ref{ex-delta214} we observe that for $g=3$ one has
\[ G_3(\Delta^0_{3,3},s) = G_3(\Delta^2_{3,0},s) + 2\times G_3(\Delta^2_{2,2},s) + 6\times G_3(\Delta^2_{1,4},s). \]
Althought note detailed here, one can check that the previous equality also holds for $g=2$ and $0\leq s \leq 2$.

These examples lead to conjecture that the higher genus Göttsche-Schroeter invariants satisfy the Abramovich-Betram formula.

\begin{conj}[Abramovich-Bertram formula] \label{conj-AB}
    Let $a,b \in\N$ and $g\geq 0$. For any $s\geq0$ one has
    \[ G_g(\Delta^0_{a,a+b},s) = \dsum_{j=0}^a \binom{b+2j}{j} G_g(\Delta^2_{a-j,b+2j},s).  \]
\end{conj}


\renewcommand{\refname}{References}
\bibliographystyle{alpha}
\bibliography{ref}

\begin{thebibliography}{GMS13}

\bibitem[AB01]{abramovich-2001-formula}
Dan Abramovich and Aaron Bertram.
\newblock The formula $12 = 10 + 2 \times 1$ and its generalizations: counting
  rational curves on {$\FF_2$}.
\newblock {\em Advances in algebraic geometry motivated by physics (Lowell, MA,
  2000), Contemp. Math.}, 276:83--88, 2001.

\bibitem[BG16]{block_refined_2016}
Florian Block and Lothar Göttsche.
\newblock Refined curve counting with tropical geometry.
\newblock {\em Compositio Mathematica}, 152(1):115--151, January 2016.

\bibitem[BJP22]{brugalle_polynomiality_2022}
Erwan Brugallé and Andrés Jaramillo-Puentes.
\newblock Polynomiality properties of tropical refined invariants.
\newblock {\em Combinatorial Theory}, 2(2), 2022.

\bibitem[BM07]{brugalle_enumeration_2007}
Erwan Brugallé and Grigory Mikhalkin.
\newblock Enumeration of curves via floor diagrams.
\newblock {\em Comptes Rendus Mathématiques de l'Académie des Sciences de
  Paris}, 345(6):329--334, September 2007.

\bibitem[BM08]{brugalle_floor_2008}
Erwan Brugallé and Grigory Mikhalkin.
\newblock Floor decompositions of tropical curves~: the planar case.
\newblock {\em Proceedings of 15th Gökova Geometry-Topology Conference}, pages
  64--90, 2008.

\bibitem[BM16]{brugalle-2016-deformation}
Erwan Brugallé and Hannah Markwig.
\newblock Deformation of tropical {Hirzebruch} surfaces and enumerative
  geometry.
\newblock {\em Journal of Algebraic Geometry}, 25:633--702, 2016.

\bibitem[BM24]{blomme24asymptotic}
Thomas Blomme and Gurvan M{\'e}vel.
\newblock Asymptotic computations of tropical refined invariants in genus 0 and
  1.
\newblock {\em arxiv:2403.17474}, 2024.

\bibitem[Bou21]{bousseau_refined_2021}
Pierrick Bousseau.
\newblock Refined floor diagrams from higher genera and lambda classes.
\newblock {\em Selecta Mathematica}, 27(43), July 2021.

\bibitem[Bru20]{brugalle-2020-invariance}
Erwan Brugallé.
\newblock On the invariance of {Welschinger} invariants.
\newblock {\em Algebra i Analiz}, 32(2):1--20, 2020.

\bibitem[BS19]{blechman-2019-refined}
Lev Blechman and Eugenii Shustin.
\newblock Refined descendant invariants of toric surfaces.
\newblock {\em Discrete \& Computational Geometry}, 62:180--208, 2019.

\bibitem[CH98]{caporaso_counting_1998}
Lucia Caporaso and Joe Harris.
\newblock Counting plane curves of any genus.
\newblock {\em Inventiones mathematicae}, 131(2):345--392, 1998.

\bibitem[DFI95]{di_francesco_quantum_1995}
Philippe Di~Francesco and Claude Itzykson.
\newblock Quantum intersection rings.
\newblock In {\em The Moduli Space of Curves}, pages 81--148, Boston, MA, 1995.
  Birkh{\"a}user Boston.

\bibitem[FM10]{fomin_labeled_2010}
Sergey Fomin and Grigory Mikhalkin.
\newblock Labeled floor diagrams for plane curves.
\newblock {\em Journal of the European Mathematical Society}, pages 1453--1496,
  2010.

\bibitem[FM11]{franz-2011-tropical}
Marina Franz and Hannah Markwig.
\newblock Tropical enumerative invariants of {$\FF_0$ and $\FF_2$}.
\newblock {\em Advances in Geometry}, 11(1):49--72, 2011.

\bibitem[GMS13]{gathmann-2013-broccoli}
Andreas Gathmann, Hannah Markwig, and Franziska Schroeter.
\newblock Broccoli curves and the tropical invariance of {Welschinger} numbers.
\newblock {\em Advances in Mathematics}, 240:520--574, 2013.

\bibitem[G{\"o}t98]{gottsche_conjectural_1998}
Lothar G{\"o}ttsche.
\newblock A conjectural generating function for numbers of curves on surfaces.
\newblock {\em Communications in Mathematical Physics}, 196(3):523--533,
  September 1998.

\bibitem[GS19]{gottsche_refined_2019}
Lothar Göttsche and Franziska Schroeter.
\newblock Refined broccoli invariants.
\newblock {\em Journal of Algebraic Geometry}, 28(1):1--41, 2019.

\bibitem[IM13]{itenberg_block-gottsche_2013}
Ilia Itenberg and Grigory Mikhalkin.
\newblock {On {Block–Göttsche} multiplicities for planar tropical curves}.
\newblock {\em International Mathematics Research Notices},
  2013(23):5289--5320, 2013.

\bibitem[KM94]{kontsevich_gromov-witten_1994}
Maxim Kontsevich and Yu~Manin.
\newblock Gromov-{Witten} classes, quantum cohomology, and enumerative
  geometry.
\newblock {\em Communications in Mathematical Physics}, 164(3):525--562, 1994.

\bibitem[Mik05]{mikhalkin_enumerative_2005}
Grigory Mikhalkin.
\newblock Enumerative tropical algebraic geometry in {$\R^2$}.
\newblock {\em Journal of the American Mathematical Society}, 18(2):313--377,
  January 2005.

\bibitem[Shu06]{shustin-2006-tropical}
Eugenii Shustin.
\newblock A tropical calculation of the {W}elschinger invariants of real toric
  del {P}ezzo surfaces.
\newblock {\em Journal of Algebraic Geometry}, 15:285--322, 2006.

\bibitem[SS18]{schroeter-2018-refined}
Franziska Schroeter and Eugenii Shustin.
\newblock Refined elliptic tropical enumerative invariants.
\newblock {\em Israel Journal of Mathematics}, 225(2):817--869, 2018.

\bibitem[SS24]{shustin-2024-refined}
Eugenii Shustin and Uriel Sinichkin.
\newblock Refined tropical invariants and characteristic numbers.
\newblock {\em arxiv:2408.08420}, 2024.

\bibitem[Tze12]{tzeng_proof_2012}
Yu-jong Tzeng.
\newblock A proof of the {G}{\"o}ttsche-{Y}au-{Z}aslow formula.
\newblock {\em Journal of Differential Geometry}, 90:439--472, 2012.

\bibitem[Vak00]{vakil-2000-counting}
Ravi Vakil.
\newblock Counting curves on rational surfaces.
\newblock {\em Manuscripta Mathematica}, 102:53--84, 2000.

\bibitem[Wel05]{welschinger_invariants_2005}
Jean-Yves Welschinger.
\newblock Invariants of real rational symplectic 4-manifolds and lower bounds
  in real enumerative geometry.
\newblock {\em Inventiones mathematicae}, 162(1):195--234, 2005.

\end{thebibliography}

\end{document}